%% file: HMgokova.tex
\documentclass[11pt]{amsart}
\usepackage{amsmath, amsthm, mathabx,amscd, amsfonts, amssymb, hyperref, mathrsfs, textcomp, epsfig, a4wide, graphicx, IEEEtrantools, tikz, verbatim, xypic}

\usetikzlibrary{matrix}

\newtheorem{thm}{Theorem}[section]
\newtheorem{prop}[thm]{Proposition}
\newtheorem{lem}[thm]{Lemma}
\newtheorem{cor}[thm]{Corollary}
\theoremstyle{definition}
\newtheorem{defn}[thm]{Definition}
\newtheorem{Exercise}[thm]{Exercise}

\theoremstyle{remark}
\newtheorem{rem}[thm]{Remark}
\newtheorem{exm}[thm]{Example}

     \RequirePackage{rotating}                   
    \def\HSt{%
       \setbox0=\hbox{$\widehat{\mathit{HS}}$}
       \setbox1=\hbox{$\mathit{HS}$}
       \dimen0=1.1\ht0
       \advance\dimen0 by 1.17\ht1
       \smash{\mskip2mu\raise\dimen0\rlap{%
          \begin{turn}{180}
              {$\widehat{\phantom{\mathit{HS}}}$}
           \end{turn}} \mskip-2mu    
                \mathit{HS}
    }{\vphantom{\widehat{\mathit{HS}}}}{}}

     \RequirePackage{rotating}                   
    \def\HMt{%
       \setbox0=\hbox{$\widehat{\mathit{HM}}$}
       \setbox1=\hbox{$\mathit{HM}$}
       \dimen0=1.1\ht0
       \advance\dimen0 by 1.17\ht1
       \smash{\mskip2mu\raise\dimen0\rlap{%
          \begin{turn}{180}
              {$\widehat{\phantom{\mathit{HM}}}$}
           \end{turn}} \mskip-2mu    
                \mathit{HM}
    }{\vphantom{\widehat{\mathit{HM}}}}{}}
    
    \newcommand{\HMb}{\overline{\mathit{HM}}}
\newcommand{\HMf}{\widehat{\mathit{HM}}}
    \newcommand{\HSb}{\overline{\mathit{HS}}}
\newcommand{\HSf}{\widehat{\mathit{HS}}}
\newcommand{\spin}{\mathfrak{s}}
\newcommand{\ztwo}{\mathbb{F}}
\newcommand{\Pin}{\mathrm{Pin}(2)}
\newcommand{\Rin}{\mathcal{R}}
\newcommand{\V}{\mathcal{V}}

\begin{document}

\title{Lectures on monopole Floer homology} 

\author{Francesco Lin}
\address{Department of Mathematics, Massachusetts Institute of Technology} 
\email{linf@math.mit.edu}

\begin{abstract}
These lecture notes are a friendly introduction to monopole Floer homology. We discuss the relevant differential geometry and Morse theory involved in the definition. After developing the relation with the four-dimensional theory, our attention shifts to gradings and correction terms. Finally, we sketch the analogue in this setup of Manolescu's recent disproof of the long standing Triangulation Conjecture. 
\end{abstract}

\maketitle

\section*{Introduction}

The present notes are a friendly introduction to monopole Floer homology for low dimensional topologists. The topic has its definitive (and essentially self-contained) reference \cite{KM} in which the whole theory is developed in detail. On the other hand, the monograph is quite scary at a first sight, both because of its size and its demanding analytical content (which might be stodgy to many people in the field). Our goal here is to explain the subject without going too deep in the details, and try to convey the key ideas involved. Of course we need to assume some background from the reader. In particular, we expect two things.
\begin{itemize}
\item A basic understanding of Seiberg-Witten theory in dimension four, following for example the classic reference \cite{Mor} (which contains much more than we require). In particular we expect the reader to have digested the differential geometry needed to write down the equations, and to have an idea on how one can use them to define invariants of smooth four manifolds with $b_2^+\geq 2$.
\item A solid understanding of Morse theory in finite dimensions, including the Morse-Witten chain complex. The reader should know how to prove a priori invariance (i.e. without referring to the isomorphism with singular homology) using continuation maps. There are many good references for this, see for example \cite{Hut} for a nice introduction and \cite{Sch} for a more thorough discussion.
\end{itemize}
Roughly speaking, the main complication is that the Seiberg-Witten equations are invariant under an $S^1$-action which is not free. In usual Morse homology (in finite dimensions), we try to understand the homology of a manifold $M$ using a Morse function $f$ on it. In our case, $M$ comes with an $S^1$-action and the goal we have in mind is to understand the $S^1$-{equivariant} homology of $M$. To do this, we will introduce a suitable model in Morse homology.
\par
Of course a basic knowledge of the cousin theory Heegaard Floer homology (\cite{OS1}, \cite{OS2}) will be helpful when dealing with the formal aspects of the theory, but we will not assume that.
\\
\par

The theory has many interesting applications in the study of low dimensional topology. Many of these are already outlined in the last Chapters of \cite{KM} and we crafted these notes so that the reader should be able to read those after digesting them. Furthermore the proof of many interesting results in Heegaard Floer homology is formally identical in our setting. For this reason we will build up towards an application which is missing in both setups, namely a disproof of the \textit{Triangulation Conjecture in higher dimensions}. This almost one-hundred-year old problem was settled by Manolescu using his Seiberg-Witten Floer homotopy type approach (\cite{Man2}). The papers \cite{Man3} and \cite{Man4} provide very nice accounts of the background of the problem. In the last few sections of these notes we will build toward the alternative (but formally identical) argument of \cite{Lin}, and we refer the reader to those surveys for a more detailed discussion of the Triangulation Conjecture itself.
\\
\par
Of course there are many sins of omission in the present lectures. Among the others:
\begin{itemize}
\item We will not be able to provide interesting examples of computations. Some of these can be obtained using the surgery exact triangle, see \cite{KMOS} and Chapter $42$ of \cite{KM}.
\item Throughout the notes, we will forget about orientations of moduli spaces and use only coefficients in $\ztwo$, the field with two elements.
\item We will not discuss the applications of this story to the gluing properties of the Seiberg-Witten invariants, which is indeed the original motivation for the definition of the Floer homology groups. This is nicely described in Chapter $3$ of \cite{KM}. Similarly, the reader can find there a discussion of local coefficients.
\item We will not discuss the beautiful non vanishing result which plays a key role in Taubes' proof of the Weinstein conjecture in dimension three (\cite{Tau}). The details of this are provided in Chapters $33-35$ in \cite{KM}.
\end{itemize}
Throughout the lectures we will provide some exercises (with hints) which are worth thinking about. The solution to most of them can be found in \cite{KM}.

\vspace{1cm}
\section{The formal picture}
We describe the structure of the invariants we will construct. Again we will only consider coefficients in $\ztwo$, the field with two elements. In these notes will focus on closed oriented connected three manifolds. To such a $Y$ we associate three $\ztwo$-vector spaces
\begin{equation*}
\HMt_{*}(Y),\quad \HMf_{*}(Y),\quad \HMb_{*}(Y)
\end{equation*}
called the monopole Floer homology groups. These are read respectively \textit{HM-to}, \textit{HM-from} and \textit{HM-bar}. These decompose according to the spin$^c$ structures on the three manifold $Y$, i.e.
\begin{equation*}
\HMt_{*}(Y)=\bigoplus_{\spin\in\mathrm{Spin}^c(Y)}\HMt_{*}(Y,\spin)
\end{equation*}
Later, we will review more in detail what spin$^c$ structures are. For the moment, one should just keep in mind that they are an affine space over $H^2(Y;\mathbb{Z})$ and
each of the $\HMt_{*}(Y,\spin)$ is a relatively $\mathbb{Z}/d\mathbb{Z}$-graded vector space, where $d$ is the integer given by the positive (and necessarily even) generator of the image of the map
\begin{align*}
H^1(Y;\mathbb{Z})&\rightarrow\mathbb{Z}\\
x&\mapsto \langle c_1(\spin)\cup x,[Y]\rangle
\end{align*}
Here $c_1(\spin)$ denotes the first Chern class of the spin$^c$ structure $\spin$. After reducing to $\mathbb{Z}/2\mathbb{Z}$ this relative grading can be enhanced to an absolute grading, so that it makes sense to talk about the even and odd components of the monopole Floer groups. Furthermore, when $c_1(\spin)$ is torsion, the relative $\mathbb{Z}$ grading can be lifted to an absolute $\mathbb{Q}$-grading.
\par
The groups are also modules over the ring
\begin{equation}\label{ring}
\Lambda_*(H_1(Y;\mathbb{Z})/\mathrm{Tors})\otimes \ztwo[U]
\end{equation}
where the action of the elements $H_1(Y;\mathbb{Z})$ has degree $-1$ while the action of $U$ has degree $-2$. The $\ztwo[U]$-action is interpreted as follows. As briefly mentioned in the introduction, the monopole Floer homology groups should be thought as $S^1$-equivariant homology groups (of an infinite dimensional space). In particular they are modules over the $S^1$-equivariant homology of the point (which is $\ztwo[U]$). This is given by the cap product induced by the inclusion of a point.
\\
\par
Roughly speaking, the Floer homology groups are the middle dimensional homology of an infinite dimensional manifold with boundary $\mathcal{B}^{\sigma}$. In particular, the three groups are respectively the homology of the space, the homology relative to the boundary and the homology of the boundary. This is the intuition behind the next result.
\begin{prop}
The monopole Floer groups fit in an exact triangle
\begin{equation}\label{LES}
\cdots \stackrel{i_*}{\longrightarrow} \HMt_{*}(Y)\stackrel{j_*}{\longrightarrow}\HMf_{*}(Y)\stackrel{p_*}{\longrightarrow} \HMb_{*}(Y)\stackrel{i_*}{\longrightarrow} \HMt_{*}(Y)\stackrel{j_*}{\longrightarrow} \cdots
\end{equation}
where the maps $i_*$ and $j_*$ are even while $p_*$ is odd. Furthermore these are $\ztwo[U]$-module maps.
\end{prop}

Of course, one can also define the monopole Floer cohomology groups
\begin{equation*}
\HMt^{*}(Y),\quad \HMf^{*}(Y),\quad \HMb^{*}(Y)
\end{equation*}
and the analogue of Poincar\'e and Lefschetz duality holds as follows.
\begin{prop}\label{poincare}
There are canonical isomorphisms of $\ztwo[U]$-modules
\begin{align*}
\HMt^{*}(Y)&\cong \HMf_{*}(-Y)\\
\HMf^{*}(Y)&\cong \HMt_{*}(-Y)\\
\HMb^{*}(Y)&\cong \HMb_{*}(-Y)
\end{align*}
for some appropriate grading shift. Here $-Y$ denotes the manifold with the opposite orientation.
\end{prop}
On cohomology the action of $\ztwo[U]$ is by cup product, so in particular multiplication by $U$ has degree $2$.
\\
\par
The simplest case is that of $S^3$. Recall that in this case there is a unique spin$^c$ structure.
\begin{prop}\label{S3}
We have the identifications as absolutely graded $\ztwo[U]$-modules:
\begin{align*}
\HMt_{*}(S^3)&\cong \ztwo[U^{-1},U]/\ztwo[U]\\
\HMf_{*}(S^3)&\cong \ztwo[U]\langle-1\rangle\\
\HMb_{*}(S^3)&\cong \ztwo[U^{-1},U].
\end{align*}
\end{prop}
The angular brackets indicate the grading shift: for a graded module $M$ we define $M\langle d\rangle$ to be the module whose homogeneous part of degree $i$ consists of the homogeneous part of degree $i-d$ of $M$. The element with highest grading in $\HMf_{*}(S^3)$ has degree $-1$.
\par
Some qualitative aspects of this computation hold more in general.
\begin{prop}\label{rational}
Let $Y$ be a rational homology sphere. Then for every spin$^c$ structure $\spin$ we have an isomorphism of graded $\ztwo[U]$-modules
\begin{equation*}
\HMb_{*}(Y,\spin)\cong \HMb_{*}(S^3)
\end{equation*}
up to grading shift. The group $\HMt_*(Y,\spin)$ vanishes in degrees low enough, and the map $i_*$ is an isomorphism in degrees high enough.
\end{prop}
In light of this result we have the identification of graded modules
\begin{equation*}
i_*(\HMb_*(Y))=\ztwo[U^{-1},U]/\ztwo[U]\langle -2h\rangle
\end{equation*}
\begin{defn}\label{froyshov}
The rational number $h$ is called the \textit{Fr\o yshov invariant} of the rational homology sphere $Y$.
\end{defn}
The Fr\o yshov invariant satisfies the following properties.
\begin{prop}
If $Y$ is a homology sphere, $h(Y)$ is an integer. This quantity is invariant under homology cobordism, i.e. if there is an oriented cobordism $W$ from $Y_-$ to $Y_+$ so that both inclusions induce isomorphism in homology then $h(Y_-)=h(Y_+)$.
\end{prop}

The final important piece is functoriality. Indeed many of the results discussed above (including invariance) follow from this. An oriented $spin^c$ cobordism $(W,\spin)$ between two three manifolds $(Y_-,\spin_-)$ and $(Y_+,\spin_+)$ gives rise to a module map
\begin{equation*}
\HMt_*(W,\spin):\HMt_{*}(Y_-,\spin_-)\rightarrow \HMt_{*}(Y_+,\spin_+).
\end{equation*}
The same statement holds for the other two versions of monopole Floer homology, and the maps commute with the maps in the long exact sequence (\ref{LES}). If we want to consider all spin$^c$ structures at the same time (which is convenient when discussing compositions), we need to consider the \textit{completed} monopole Floer groups
\begin{equation}\label{total}
\HMt_{\bullet}(Y),\quad \HMf_{\bullet}(Y),\quad \HMb_{\bullet}(Y),
\end{equation}
which are obtained by taking the completion with respect to a certain filtration defined by the gradings. For example, $\HMf_{\bullet}(S^3)$ is identified with $\ztwo[[U]]\langle-1\rangle$. The issue is that infinitely many spin$^c$ structures on $W$ might induce a non-trivial map. In particular, the sum of all these maps might not be well defined. Nevertheless, a cobordism induces a well defined map
\begin{align}\label{totalmap}
\begin{split}
\HMt_{\bullet}(W)&:\HMt_{\bullet}(Y_-)\rightarrow \HMt_{\bullet}(Y_+)\\
\HMt_{\bullet}(W)&=\bigoplus_{\spin\in\mathrm{Spin}^c(W)} \HMt_{\bullet}(W,\spin)
\end{split}
\end{align}
The induced maps compose in a functorial way as follows.
\begin{prop}\label{functoriality}
Given a cobordism $W_1$ from $Y_0$ to $Y_1$ and a cobordism $W_2$ from $Y_1$ to $Y_2$, we have that
\begin{equation*}
\HMt_{\bullet}(W_2\circ W_1)=\HMt_{\bullet}(W_2)\circ \HMt_{\bullet}(W_1).
\end{equation*}
where $W_2\circ W_1$ is the composite cobordism (from $Y_0$ to $Y_2$).
\end{prop}
The key point of functoriality is the nice interplay between three and four dimensional Seiberg-Witten theory. We will discuss this in quite detail.

\vspace{1cm}

\vspace{1cm}
\section{The Seiberg-Witten equations on a three manifold}\label{SWeq}
In this section we discuss the differential geometry involved in our construction. In particular, as mentioned in the Introduction, we will associate to a (closed, oriented, and connected) three manifold $Y$ (together with extra data) a manifold $M$ equipped with an $S^1$-action, and an $S^1$ invariant function $f$.
\\
\par
\textbf{Spin$^c$ structures. }We start by discussing spin$^c$ structures. There are many ways to introduce them (see also Exercise \ref{spinc}), but for our purposes the most natural one is the following.
\par
Suppose $Y$ is an oriented Riemannian $3$-manifold. Then a spin$^c$ structure is a hermitian rank $2$ bundle $S\rightarrow Y$ together with a Clifford multiplication
\begin{equation*}
\rho: TY\rightarrow \mathrm{Hom}(S,S).
\end{equation*}
The latter is a bundle map satisfying $\rho(v)^2=-|v|^2 1_S$ for each $v\in TY$. Very concretely, this means that given any oriented frame $e_1,e_2,e_3$ at a point $y$, we can find a basis of $S_y$ so that $\rho(e_i)$ is the Pauli matrix $\sigma_i$:
\begin{equation}\label{pauli}
\sigma_1=\begin{pmatrix}
i & 0\\
0 & -i
\end{pmatrix}, \quad
\sigma_2=\begin{pmatrix}
0 & -1\\
1 & 0
\end{pmatrix},\quad
\sigma_3=\begin{pmatrix}
0 & i\\
i & 0
\end{pmatrix}.
\end{equation}
These form standard basis of $\mathfrak{su}(S)$, the space of traceless skew-adjoint endomorphisms of $S$.
\begin{Exercise}
Check that $\rho$ is a bundle isometry if $\mathfrak{su}(2)$ is equipped with the hermitian product $\frac{1}{2}\mathrm{tr}(ab^*)$.
\end{Exercise}

We call $S$ the \textit{spinor bundle} and its sections \textit{spinors}. Spin$^c$ structures always exist: indeed, $TY$ is trivial for oriented three manifolds so we can pick a global trivialization $e_1,e_2,e_3$ and make it act on the trivial bundle $\mathbb{C}^2\times Y$ globally via the Pauli matrices.
\begin{Exercise}\label{spinc}
Some readers might be more familiar to the definition of a spin$^c$ structure as equivalence classes of non vanishing vector fields up to homotopy outside a ball. This correspond to ours as follows. Fix a unit length spinor $\Psi$. Then there is a unique vector field $X(\Psi)$ such that at each point the $i$ and $-i$ eigenspaces of $\rho(X(\Psi))$ are respectively $\mathbb{C}\Psi$ and its orthogonal complement. Show that different choices of the spinor give rise to vector fields homotopic outside a ball. What is the inverse of this map?
\end{Exercise}
The first Chern class $c_1(\spin)$ of the spin$^c$ structure is the first Chern class of the spinor bundle $S$. We say that a spin$^c$ structure is \textit{torsion} if $c_1(\spin)$ is torsion, and \textit{non-torsion} otherwise. We also have the following complete classification of spin$^c$ structures.
\begin{lem}\label{spincclass}
The set of spin$^c$ structures on $Y$ is an affine space over $H^2(Y,\mathbb{Z})$.
\end{lem}
\begin{proof}
There is a natural bijection between $H^2(Y,\mathbb{Z})$ and the group of complex line bundles (with tensor product as operation) given by the first Chern class. Suppose we are given a spin$^c$ structure $(S_0,\rho_0)$. Then for a given hermitian line bundle $L$, we can define the pair
\begin{equation*}
(S_0\otimes L, \rho_0\otimes 1_L).
\end{equation*}
which is readily checked to be a spin$^c$ structure. The next exercise contains hints to prove the stated result.
\end{proof}
\begin{Exercise}\label{irr}
Using the fact that the representation of $\mathfrak{su}(2)$ on $S$ is irreducible, check that this map is actually a bijection. Hint: show that given two spin$^c$ structures $(S,\rho)$ and $(S',\rho')$, the space of bundle maps intertwining the Clifford multiplications is a complex line bundle.
\end{Exercise}

The Clifford multiplication can be first extended to cotangent vectors via the identification with tangent vectors using the metric, and then to forms using the rule
\begin{equation*}
\rho(\alpha\wedge\beta)=\frac{1}{2}\left( \rho(\alpha)\rho(\beta)+(-1)^{\mathrm{deg}(\alpha)\mathrm{deg}(\beta)}\rho(\beta)\rho(\alpha)\right).
\end{equation*}
Finally, we can extend it to complex forms.
\begin{Exercise}\label{hodge}
The Riemannian metric induces the Hodge star operator $\ast$. Show that $\rho(\ast\alpha)=-\rho(\alpha)$.
\end{Exercise}

\vspace{0.5cm}
\textbf{The configuration space and the Dirac operator. }There is a natural class of connections on the bundle $S$ which is compatible with both the hermitian metric and the Clifford multiplication. We say that a connection $B$ on $S$ is a \textit{spin$^c$ connection} if the associated covariant derivative $\nabla_B$ satisfies
\begin{equation*}
\nabla_B (\rho(X)\Psi)= \rho(\nabla X)\Psi+ \rho(X)\nabla_B\Psi
\end{equation*}
where $X$ is any vector field and $\Psi$ is any spinor. Here $\nabla$ is the Levi-Civita connection on $TY$. The space of spin$^c$ connections is an affine space over $\Omega^1(Y;i\mathbb{R})$, again by the irreducibility of the action of $\rho$ (see Exercise \ref{irr}). It is convenient in many cases to pass to the connection $B^t$ induced on the determinant line bundle $\mathrm{det} (S)$, so that we are working with genuine imaginary valued one forms: if $\tilde{B}-B=b\otimes 1_S$ then $\tilde{B}^t-B^t=2b$.
\\
\par
The \textit{configuration space} of the pair $(Y,\spin)$ is then defined to be the space
\begin{equation*}
\mathcal{C}(Y,\spin)=\{ (B,\Psi)\}
\end{equation*}
consisting of pairs of a spin$^c$ connection and a spinor.
\\
\par
Connections and spinors interact with each other via the Dirac operator $D_B$, which is the composition
\begin{equation*}
\Gamma(S)\stackrel{\nabla_B}{\longrightarrow} \Gamma(T^*X\otimes S)\longrightarrow \Gamma(S)
\end{equation*}
where the latter is the Clifford multiplication (extended to one forms, see the discussion following Exercise \ref{irr}). This is a first order self-adjoint elliptic operator, hence the spectral theorem implies the following.
\begin{lem}\label{spectral}
The eigenvalues of $D_B$ form a discrete subset of $\mathbb{R}$ which is infinite in both directions. Furthermore there is a complete orthonormal basis of eigenvectors for $D_B$.
\end{lem}
A good way to remember its properties is to go two dimensions down and think about the standard Dirac operator on the circle $S^1=\mathbb{R}/2\pi\mathbb{Z}$, which is simply
\begin{equation}\label{dirac1}
-i\frac{d}{d\theta}: C^{\infty}(S^1; \mathbb{C})\rightarrow C^{\infty}(S^1; \mathbb{C}).
\end{equation}
In this case we can identify the basis of eigenvectors as the usual Fourier basis $\{e^{in\theta}\}_{n\in\mathbb{Z}}$.
\begin{Exercise}
Consider the flat torus $\mathbb{T}=S^1\times S^1\times S^1$. Consider the trivial spin$^c$ structure with the trivial connection $B_0$. Write down $D_{B_0}$ and compute its spectrum.
\end{Exercise}
\vspace{0.5cm}

\textbf{The gauge group. }The group of automorphisms of the spin$^c$ structure $\spin$, also known as \textit{gauge group}, is given by
\begin{equation*}
\mathcal{G}(Y,\spin)=\{u: Y\rightarrow S^1\}.
\end{equation*}
A gauge transformation $u$ acts on a connection by pull-back and on a spinor by multiplication, in formulas
\begin{equation*}
u\cdot (B,\Psi)=(B-u^{-1}du, u\cdot\Psi).
\end{equation*}
It is straightforward to see that if the spinor $\Psi$ is not zero, then its stabilizer under the gauge group of the configuration $(B,\Psi)$ is trivial. On the other hand the stabilizer of a configuration $(B,0)$ is given by the constant gauge transformations, so it is identified with $S^1$. 
\begin{defn}\label{irrrred}
We call the configurations of the first kind \textit{irreducibles}, while the configurations of the second kind \textit{reducible}.
\end{defn}
\par
If we fix a basepoint $y_0$ in $Y$ and consider the subgroup $\mathcal{G}_0(Y,\spin)$ of gauge transformations which are $1$ at $y_0$, we see that this acts freely on the configuration space $\mathcal{C}(Y,\spin)$. The quotient of this action, the \textit{based moduli space of configurations} $\mathcal{B}_o(Y,\spin)$ is the `manifold' $M$ on which we will perform Morse theory. It carries the residual action of the quotient $\mathcal{G}(Y,\spin)/\mathcal{G}_0(Y,\spin)=S^1$.
\par
As $S^1$ is a $K(\mathbb{Z},1)$, the component group of the gauge group $\mathcal{G}(Y,\spin)$ is naturally identified with $H^1(Y;\mathbb{Z})$. The elements corresponding to the trivial class are exactly those that can be written as $e^\xi$ for some imaginary valued function $\xi$.
\begin{Exercise}\label{cohom}
Show that the cohomology class of a map $u$ is represented by the real one form $1/(2\pi i)u^{-1}du$.
\end{Exercise}
\vspace{0.5cm}

\textbf{The Chern-Simons-Dirac functional.}
Finally, we introduce the $S^1$-invariant function $f$ that will be used to compute Morse homology. For a fixed base connection $B_0$ the Chern-Simons-Dirac functional is defined to be
\begin{equation*}
\mathcal{L}(B,\Psi)=-\frac{1}{8}\int_Y (B^t-B^t_0)\wedge( F_{B^t}+F_{B^t_0})+\frac{1}{2}\int_Y \langle D_B\Psi,\Psi\rangle \mathrm{d}vol.
\end{equation*}
Here $F_{B^t}$ is the curvature of the connection $B^t$, which is an imaginary valued two form. The value of $\mathcal{L}$ is a real number because the Dirac operator $D_B$ is self-adjoint.
\par
The configuration space is an affine space and its tangent space at each point is identified with $\Omega^1(Y;i\mathbb{R})\times \Gamma(S)$, which has a natural $L^2$ inner product induced by the Riemannian metric on $Y$ and the hermitian metric on $S$. One can compute the (formal) gradient of the functional with respect to this metric, which is
\begin{equation}\label{grad}
\mathrm{grad}\mathcal{L}(B,\Psi)=(\frac{1}{2}\ast F_{B^t}+\rho^{-1}(\Psi\Psi^*)_0, D_B\Psi).
\end{equation}
Here $\ast$ denotes the Hodge star and $(\Psi\Psi^*)_0$ is the traceless part of the hermitian endomorphism $\Psi\Psi^*$. In coordinates, if $\Psi$ has components $(\alpha,\beta)$ then
\begin{equation*}
(\Psi\Psi^*)_0=\begin{pmatrix}
\frac{1}{2}(|\alpha|^2-|\beta|^2) & \alpha\bar{\beta}\\
\bar{\alpha}\beta & \frac{1}{2}(|\beta|^2-|\alpha|^2). 
\end{pmatrix}
\end{equation*}
Its inverse image under the Clifford multiplication $\rho$ is an imaginary valued one form.
\\
\par
It is important to remark that the Chern-Simons-Dirac functional is not generally fully gauge invariant. One can check using Exercise \ref{cohom} and the fact that the $-1/(2\pi i)F_{B^t_0}$ represents the first Chern class $c_1(S)$ that
\begin{equation*}
\mathcal{L}(u\cdot(B,\Psi))-\mathcal{L}(B,\Psi)=2\pi^2([u]\cup c_1(S))[Y].
\end{equation*}
We see already the big difference between torsion spin$^c$ structures (for which the functional is fully gauge invariant, hence it descends to the moduli space of configurations) and non-torsion spin$^c$ structures, for which the functional is well defined only as a circle valued function.
\vspace{0.5cm}

\textbf{Seiberg-Witten monopoles. }The equations we will be interested in understanding throughout these notes are the gradient flow equations for the Chern-Simons-Dirac functional
\begin{equation}\label{gradfloweq}
\frac{d}{dt}(B(t),\Psi(t))=-\mathrm{grad}\mathcal{L}(B(t),\Psi(t))
\end{equation}
for a path $(B(t),\Psi(t))$ of configurations in $\mathcal{C}(Y,\spin)$. The critical points of the functional $\mathcal{L}$, or, equivalently, the solutions of the system
\begin{align*}
\frac{1}{2}\ast F_{B^t}+\rho^{-1}(\Psi\Psi^*)_0&=0\\
D_B\Psi&=0
\end{align*}
are called \textit{monopoles}. In light of Exercise \ref{hodge}, the first equation is equivalent to
\begin{equation*}
\frac{1}{2}\rho(F_{B^t})-(\Psi\Psi^*)_0=0.
\end{equation*}
There is a simple description of the reducible monopoles (recall Definition \ref{irrrred}).
\begin{prop}\label{redsol}
If the spin$^c$ structure is not torsion, there are no reducible monopoles. If the spin$^c$ structure is torsion, the reducible monopoles can be identified with
\begin{equation*}
H^1(Y; i\mathbb{R})/2\pi i H^1(Y;\mathbb{Z}),
\end{equation*}
the torus of flat connections up to gauge.
\end{prop}
\begin{proof}
Recall that for any connection $B$ on $S$ the form $-1/(2\pi i) F_{B^t}$ is a de Rham representative of $c_1(S)$. As $\Psi$ vanishes, the solutions are simply flat connections $B^t$ up to gauge. If the bundle admits a flat connection then its first Chern class is torsion, as $-1/(2\pi i) F_{B^t}$ is identically zero. Suppose now $c_1(S)$ is torsion. Fix a base connection $B_0$. The curvature of $B_0^t+2b$ is $F_{B_0^t}+2db$ and as $F_{B_0^t}$ is exact we can find a $b$ so that the quantity vanishes. All other flat connections are obtained by adding closed forms. Now from Exercise \ref{cohom} gauge transformations act by the addition of closed forms whose de Rham class is in $2\pi i H^1(Y;\mathbb{Z})$, so the result follows.
\end{proof}

Irreducible solutions are much harder to describe and are strongly dependent of the underlying geometry and topology. We have for example the following result.
\begin{prop}\label{weitz}
Suppose the scalar curvature of $Y$ is non negative at each point. Then there are no irreducible solutions.
\end{prop}
\begin{proof}
Recall the Weitzenb\"ock formula
\begin{equation*}
D_B^2 \Psi= \nabla^*_{B}\nabla_B \Psi + \frac{1}{2}\rho(F_{B_t})\Psi +\frac{1}{4} s\Psi,
\end{equation*}
where $\nabla^*_B$ is the formal adjoint of $\nabla_B$ and $s$ is the scalar curvature. If $(B,\Psi)$ is a solution, we obtain by substituting the identity
\begin{equation*}
\nabla^*_{B}\nabla_B \Psi + (\Psi\Psi^*)_0\Psi +\frac{1}{4} s\Psi=0.
\end{equation*}
If we take the $L^2$ inner product with $\Psi$ we obtain
\begin{equation*}
\| \nabla_B \Psi \|^2+\frac{1}{2}\|\Psi\|_{L^4}+\langle s\Psi, \Psi\rangle_{L^2}=0.
\end{equation*}
As the scalar curvature is non negative, the last term is non negative, which implies that the spinor vanishes. So there cannot be irreducible solutions.
\end{proof}
\begin{Exercise}
Check that $\|(\Psi\Psi^*)_0\|_{L^2}$ is $\|\Psi\|_{L^4}$.
\end{Exercise}

The Weitzenb\"ock formula is the key feature that makes Seiberg-Witten theory more tractable that the instanton counterpart. The arguments as the one mentioned above are usually referred to as the \textit{Bochner technique}. Given a connection $B$ we can write two Laplacians, $D_B^2$ and $\nabla_B^*\nabla_B$, which have by construction the same second order part. The point is that they also have the same first order part, so they only differ by some pointwise defined operator.
Furthermore, this can be identified in terms of relevant geometric data. Notice that the sign in the first of the equations has a fundamental role. Similar ideas are behind the proof of the following key result.
\begin{prop}\label{compactness}
The space of critical points of $\mathrm{grad}\mathcal{L}$ up to gauge is compact.
\end{prop}

\vspace{1cm}
\section{Blowing-up and Morse theory with boundary}
Our goal is to apply the ideas of finite dimensional Morse theory to the Chern-Simons-Dirac functional $\mathcal{L}$ in order to define homology groups which are invariants of the three manifold we started with. In finite dimension, one proceeds by perturbing the function $f$ on $M$ to a Morse-Smale one, and define a chain complex which generated by the critical points and for which the differential counts isolated trajectories.
\par
There is a main complication in our setting, namely the fact that the function $f:M\rightarrow \mathbb{R}$ is invariant under an $S^1$-action which has non empty fixed point set $M^{S^1}$. In particular the quotient by the $S^1$ action is not necessarily a smooth manifold. If we want to preserve this symmetry in our perturbations (so that we do not lose the extra information), we have to take a different approach.
\par
The most naive one is to restrict our attention to the complement of the fixed point set $M^*=M\setminus M^{S^1}$ on which the action is free. This allows us to consider the usual Morse homology of the quotient $M^*/S^1$. This approach (which was taken at the early stages of the development of the theory) has two main drawbacks:
\begin{itemize}
\item As we will see, a lot of information is lost by throwing the reducibles away.
\item More importantly, the resulting groups are \textit{not} invariants of the three manifold, i.e. they depend on the particular choice of metric and perturbation. This is clear from the following simple finite dimensional model. Take $\mathbb{C}$ with the $S^1$ action given by complex multiplication, whose only fixed point is the origin. Consider for $c\in\mathbb{R}$ the Morse function
\begin{equation*}
f(z)=c\|z\|^2+\|z\|^4.
\end{equation*}
For positive $c$ the origin is the only critical point, while for negative $c$ we also have the circle of critical points $\|z\|=\sqrt{-c/2}$. In particular if $c$ is positive there are no critical points in $\mathbb{C}^*/S^1$, while if $c$ is negative there is exactly one. So the homologies are not isomorphic. 
\end{itemize}
\vspace{0.5cm}
\textbf{Blowing up. }The way these issues are solved in Kronheimer and Mrowka's approach is by blowing up the manifold $M$. In the simplest case of $\mathbb{C}$ described above this is nothing but passing to polar coordinates. Indeed, we can identify
\begin{equation*}
\mathbb{C}\equiv\mathbb{R}^{\geq0}\times S^1/(\{0\}\times S^1)
\end{equation*}
with coordinates $(r,\phi)$. The blow up of $\mathbb{C}$ is then defined to be
\begin{equation*}
\mathbb{C}^{\sigma}=\mathbb{R}^{\geq0}\times S^1
\end{equation*}
which comes with the obvious blow down map
\begin{equation*}
\pi:\mathbb{C}^{\sigma}\rightarrow \mathbb{C}.
\end{equation*}
This map is a diffeomorphism on the locus where $r>0$. In analogy with Definition \ref{irrrred}, we call this the irreducible locus. The $S^1$-action extends naturally to $\mathbb{C}^{\sigma}$ and the quotient $\mathbb{C}^{\sigma}/S^1$ is identified with $\mathbb{R}^{\geq0}$. Notice that this is a manifold with boundary.
\par
Similarly, we can do the same with the action of $S^1$ on $\mathbb{C}^n$ by complex multiplication. In this case $(\mathbb{C}^n)^{\sigma}/S^1$ is $\mathbb{R}^{\geq0}\times \mathbb{C}P^{n-1}$. More generally, suppose that $S^1$ acts on a Riemannian manifold $M$ by isometries so that:
\begin{itemize}
\item the stabilizer of each point is either $\{1\}$ or the whole $S^1$; 
\item the fixed point set is a smooth manifold $P$.
\end{itemize}
At a point $p$ in $P$ the normal fiber $N_p$ is naturally a complex vector space with an $S^1$-action as in the model discussed above. We can then construct the blow up $M^\sigma$ by replacing it with $N_p^{\sigma}$ at each point. Again this has a natural free $S^1$-action and the quotient $M^\sigma/S^1$ is a smooth manifold with boundary.

\begin{figure}
  \centering
\def\svgwidth{\textwidth}
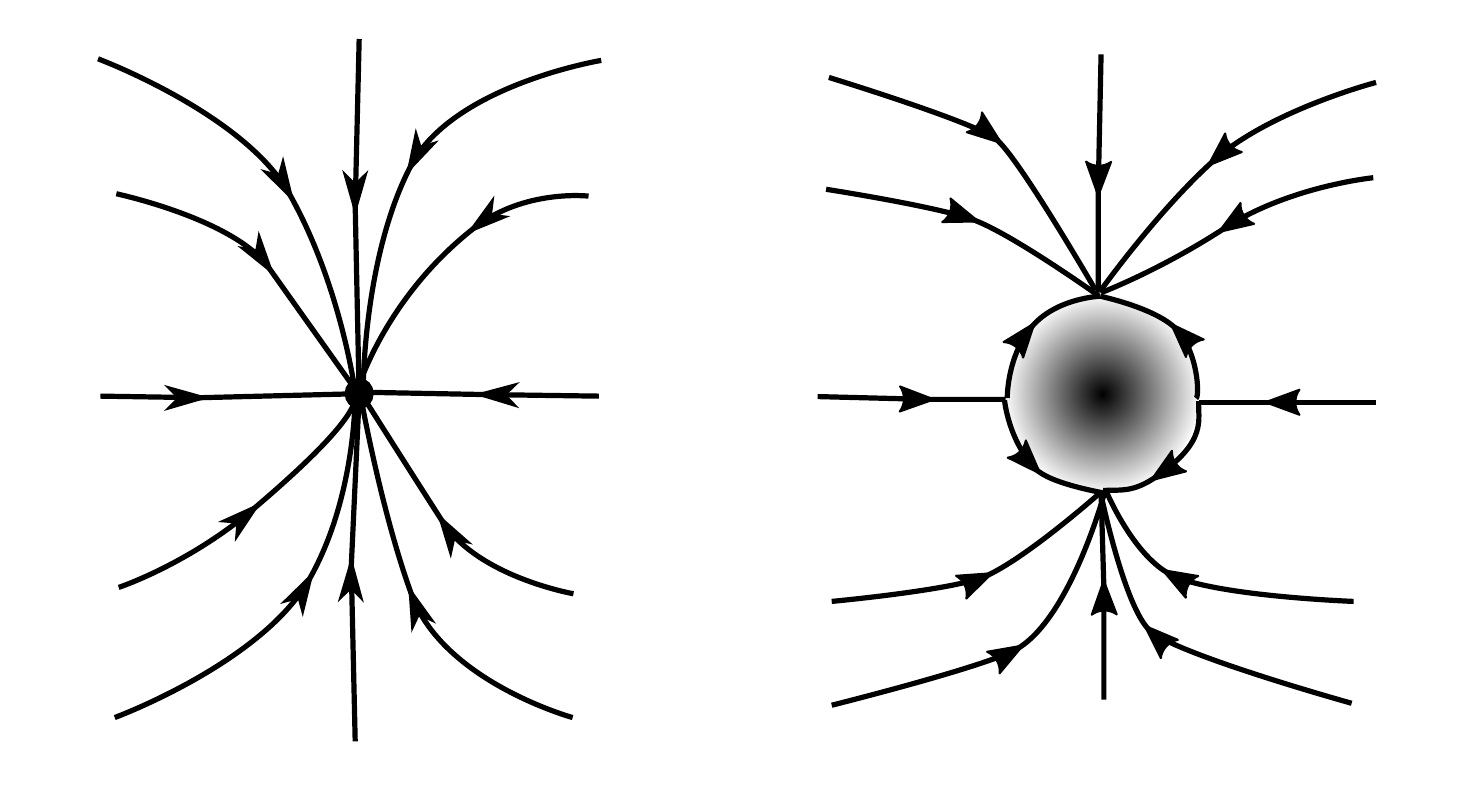
    \caption{A linear flow in $\mathbb{R}^2$ and the corresponding vector field induced in the blow up. Notice that this picture does not have the $S^1$-symmetry.}
    \label{blowgrad}
\end{figure} 

The next key observation is that if we are given a smooth $S^1$-invariant function $f:M\rightarrow\mathbb{R}$ then the pull-back of vector field $\mathrm{grad}f$ on $M^*$ naturally extends to a smooth vector field $(\mathrm{grad}f)^{\sigma}$ on $M^{\sigma}$ (see Figure \ref{blowgrad}). This is the vector field we will use to define the Morse homology groups. As our blow-up construction is performed fiberwise, the following example clarifies well what is going on. Consider a self-adjoint linear map $L$ on $\mathbb{C}^n$. Consider the $S^1$-invariant function on $\mathbb{C}^n$ given by
\begin{equation}\label{fuct}
f(z)=\frac{1}{2}\langle z, Lz\rangle
\end{equation}
whose gradient is $\mathrm{grad}f(z)=Lz$. On $(\mathbb{C}^n)^*$ this can be written in polar coordinates $(r,\phi)$ (or, equivalently, in the blow up $\mathbb{R}^{>0}\times S^{n-1}$) as
\begin{equation}\label{blowform}
\mathrm{grad}f(r,\phi)=(\Lambda(\phi)r,L\phi-\Lambda(\phi)\phi)
\end{equation} 
where $\Lambda(\phi)$ is the quadratic function $\langle \phi,L\phi\rangle$. The second component is a vector tangent to $S^{n-1}$ at $\phi$, and is sent to the part of $Lz$ tangent to the sphere of radius $r$ via the blow down map. The formula (\ref{blowform}) above makes sense on the whole $(\mathbb{C}^n)^{\sigma}$ and provides the required smooth extension.
\begin{Exercise}
Verify the formula in equation (\ref{blowform}).
\end{Exercise}
\begin{rem}
It is important to notice that $(\mathrm{grad}f)^{\sigma}$ is \textit{not} the gradient of the function on $M^\sigma$ induced by $f$ for any natural choice of the metric.
\end{rem}  

\vspace{0.5cm}
\textbf{Morse homology with boundary. }There are a few complications to deal with in our approach. These are well described in Figure $3$.
\begin{itemize}
\item The vector field $(\mathrm{grad}f)^{\sigma}$ is tangent to the boundary. In particular there are two kinds of critical points in the boundary: the outward normal direction can be stable or unstable. We call these critical points respectively \textit{stable} and \textit{unstable}. We call the critical points in the interior \textit{irreducible}.
\item The flow is not necessarily Morse-Smale. This is because the unstable manifold of a stable critical point and the stable manifold of an unstable critical point are always contained in the boundary. In particular if they have non trivial intersection, this cannot be transverse. We call such a pair \textit{boundary obstructed}.
\item A consequence of the facts above is that a sequence of trajectories in a one dimensional moduli space can break into three components.
\item The trajectories between an unstable and a stable critical points form a manifold with boundary, the latter consisting of the trajectories lying inside the boundary.
\item Finally, we are not dealing with a gradient flow.
\end{itemize}

\begin{figure}
  \centering
\def\svgwidth{0.5\textwidth}
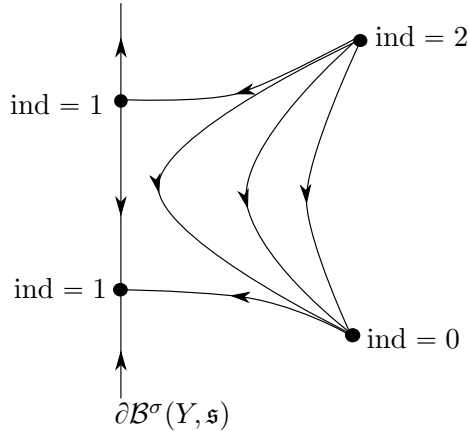
    \caption{In the picture, there is a trajectory in the boundary between critical points of the same index. Correspondingly, there is a one dimensional family of trajectories of Morse trajectories limiting to a broken trajectory with three components.}
    \label{notsmale}
\end{figure} 

The following exercise has fundamental importance for the rest of the notes. We will come back to more aspects of this example later.

\begin{Exercise}\label{crit}
Suppose that the self-adjoint operator $L$ defining the function (\ref{fuct}) has simple spectrum (i.e. one dimensional eignspaces) and is invertible. Show that:
\begin{itemize}
\item the critical points of the flow on $(\mathbb{C}^n)^{\sigma}/S^1$ correspond to the eigenvalues of $L$;
\item a critical point is stable if and only if it corresponds to a positive eigenvalue.
\end{itemize}
\end{Exercise}

Even though there are many complications, all of them can be overcome. We now  introduce a suitable notion of transversality.
\begin{defn}\label{regularity}
We say that an $S^1$-invariant function $f$ on $M$ is \textit{regular} if the critical points of $(\mathrm{grad}f)^{\sigma}$ on $M^{\sigma}/S^1$ are Morse and
\begin{itemize}
\item for each non boundary obstructed pair $(a,b)$ of critical points, the unstable manifold of $a$ and the stable manifold of $b$ intersect transversely;
\item for each boundary obstructed pair $(a,b)$ of critical points, the unstable manifold of $a$ and the stable manifold of $b$ intersect transversely \textit{within the boundary};
\end{itemize}
In particular, the flow restricted to the boundary is Morse-Smale.
\end{defn}

\begin{Exercise}
Compute the dimension of the space of trajectories $M(a,b)$ between $a$ and $b$ in terms of their index and type (stable, unstable, irreducible). Does some combination of types imply immediately that these spaces are empty?
\end{Exercise}

Notice that we can adapt these definitions to any Riemannian manifold with boundary $B$ and vector field $v$ obtained as follows (see Figure \ref{disc} for an explicit example). The manifold $B$ is obtained as the quotient of a Riemannian manifold $\tilde{B}$ by an isometric involution $\iota$ whose fixed point set is a codimension one smooth manifold. The vector field $v$ is the restriction of the gradient of an $\iota$-invariant function. We will refer to this model from now on, but all the discussion will also apply to the original case.

\begin{figure}
  \centering
\def\svgwidth{0.3\textwidth}
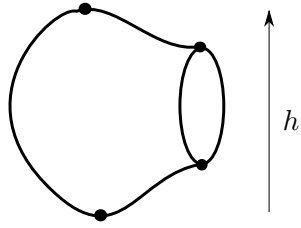
    \caption{The gradient of the height function on this half sphere has four critical points and is tangent to the boundary.}
    \label{disc}
    \end{figure} 

Under the transversality assumptions of Definition \ref{regularity} (which can be achieved generically), we define the vector spaces $C^o_k$, $C^s_k$ and $C^u_k$ generated respectively by the irreducible, stable and unstable critical points of index $k$. We will drop the index to mean the direct sum of all of them. We can define the linear operator
\begin{equation*}
\partial^o_o: C^o_k\rightarrow C^o_{k-1}
\end{equation*}
obtained by counting trajectories between irreducible critical points
\begin{equation*}
\partial^o_o a=\sum_{\substack{b\in C^o  \\ \mathrm{ind}(a)-\mathrm{ind}(b)=1}} \#(\check{M}(a,b))\cdot b.
\end{equation*}
Here $\check{M}(a,b)$ denotes the finite set of \textit{unparametrized} trajectories from $a$ to $b$. Notice that these all lie in the interior because the limit points are irreducible. Similarly by counting isolated interior trajectories we can define the three operators $\partial^o_s,\partial^u_o$ and $\partial^u_s$, all of which drop the index by one. Here the apex indicates the domain and the index the codomain. All these trajectories are contained in the interior.
\par
Similarly we can define the operators $\bar{\partial}^s_s,\bar{\partial}^u_u,\bar{\partial}^u_s$ and $\bar{\partial}^s_u$ by counting isolated trajectories contained entirely in the boundary. The first two operators drop the index by one, the third by two while the last one (which corresponds to the boundary obstructed case) leaves the index unchanged. Notice that the maps $\partial^u_s$ and $\bar{\partial}^u_s$ count different trajectories (and indeed shift the degree differently).
\\
\par
Define now the graded vector spaces $\check{C}_*$, $\hat{C}_*$ and $\bar{C}_*$ whose graded parts are
\begin{align}\label{chain}
\begin{split}
\check{C}_k&=C^o_k\oplus C^s_k\\
\hat{C}_k&=C^o_k\oplus C^u_k\\
\bar{C}_k&=C^s_k\oplus C^u_{k+1}.
\end{split}
\end{align}
We define the differentials on them given in components as
\begin{equation}\label{diff}
\check{\partial}=\begin{bmatrix}
\partial^o_o & \partial^u_o\bar{\partial}^s_u\\
\partial^o_s & \bar{\partial}^s_s+\partial^u_s\bar{\partial}^s_u
\end{bmatrix}
\qquad
\hat{\partial}=\begin{bmatrix}
\partial^o_o & \partial^u_o\\
\bar{\partial}^s_u\partial^o_s & \bar{\partial}^u_u+\bar{\partial}^s_u\partial^u_s
\end{bmatrix}
\qquad
\bar{\partial}=\begin{bmatrix}
\bar{\partial}^s_s& \bar{\partial}^u_s\\
\bar{\partial}^s_u&\bar{\partial}^u_u
\end{bmatrix}
\end{equation}
It is easy to check that these maps have degree $-1$, and we have in fact the following.
\begin{lem}
The three pairs $(\check{C}_*,\check{\partial}), (\hat{C}_*,\hat{\partial})$ and $(\bar{C}_*,\bar{\partial})$ are chain complexes.
\end{lem}
\begin{proof}
As in usual Morse theory the fact that the differential squares to zero follows from identifying the boundary strata of the compactified one dimensional spaces of trajectories. In our case there are several cases to be considered. For example, by considering the trajectories connecting two irreducible critical points whose indices differ by two we obtain the identity
\begin{equation}\label{osuo}
\partial^o_o\partial^o_o+ \partial^u_o\bar{\partial}^s_u\partial^o_s.
\end{equation}
The first term corresponds to the standard interior broken trajectories, while the second counts triple broken trajectories as in Figure $3$. Details can be found in Chapter $22$ of \cite{KM}, but it is a good exercise to work out.
\end{proof}

The homology groups of these chain complexes can be identified in terms of the usual homology groups of the manifold with boundary $B$. 

\begin{prop}
The homologies $H_*(\check{C},\check{\partial}), H_*(\hat{C},\hat{\partial})$ and $H_*(\bar{C},\bar{\partial})$ are independent of the choice of Morse function and metric and are isomorphic to $H_*(B), H_*(B,\partial B)$ and $H_*(\partial B)$.
\end{prop}

\begin{Exercise}\label{disk}Compute the three homologies of the disk using the vector field of the function depicted in Figure \ref{disc}.
\end{Exercise}

As in usual Morse theory, the independence statement is stronger than just the invariance of the isomorphism class of the group. Indeed, given two admissible choices $(f,g)$ and $(f',g')$ of Morse function and metric there is a well defined isomorphism
\begin{equation}\label{continuation}
\Phi((f,g),(f',g')): H_*(\check{C},\check{\partial},f,g)\rightarrow H_*(\check{C},\check{\partial},f,g)
\end{equation} 
so that
\begin{align*}
\Phi((f',g'),(f'',g''))\circ\Phi((f,g),(f',g'))&=\Phi((f,g),(f'',g''))\\
\Phi((f,g),(f,g))&=\mathrm{Id}
\end{align*}
These isomorphisms are constructed via continuation maps. The existence of the continuation maps provides an a priori proof of invariance (i.e. not referring to the isomorphism with singular homology).
\\
\par
Returning to the case of a space $M$ with an $S^1$-action, as mentioned in the Introduction a natural homology theory to consider is $S^1$-equivariant homology. One can check that our model provides an appropriate Morse-homological approach to its computation. In particular, it is shown in Section $2.6$ in \cite{KM} that $H^{S^1}_{\leq k}(M)$ is determined from the construction above applied to $M\times \mathbb{C}^N$ for $N$ big enough (and appropriate choice of Morse function on $\mathbb{C}^N$).

\begin{Exercise}\label{finiteLES}
Prove the usual long exact sequence of the triple
\begin{equation*}
\cdots\longrightarrow H_k(\partial B)\longrightarrow H_k(B)\longrightarrow H_k(X,\partial B)\longrightarrow H_{k-1}(\partial B)\longrightarrow \cdots
\end{equation*}
by constructing suitable chain maps in our model of Morse homology with boundary.
\end{Exercise}

\vspace{1cm}
\section{Floer homology}
We now apply the the blow-up construction discussed above to the geometric setup of Section \ref{SWeq}. First of all, the \textit{blown-up configuration space} is the space $\mathcal{C}^{\sigma}(Y,\spin)$ of triples $(B,r,\psi)$ where
\begin{itemize}
\item $B$ is a spin$^c$ connection;
\item $r$ is a non negative real number;
\item $\psi$ is a spinor with $\|\psi\|_{L^2}=1$.
\end{itemize}
This is the space obtained by passing to polar coordinates in the (infinite dimensional) vector space of spinors. It comes with the blow down map
\begin{align*}
\pi: \mathcal{C}^{\sigma}(Y,\spin)&\rightarrow \mathcal{C}(Y,\spin)\\
(B,r,\psi)&\mapsto (B,r\psi)
\end{align*}
which is a diffemorphism when restricted to irreducible configurations. The gauge group $\mathcal{G}(Y,\spin)$ naturally acts on $\mathcal{C}^{\sigma}(Y,\spin)$ and the action is now free. The quotient $\mathcal{B}^{\sigma}(Y,\spin)$ is an infinite dimensional manifold with boundary consisting of the configurations for which $r$ is zero. The pull-back of the gradient of the Chern-Simons-Dirac functional (\ref{grad}) extends to the whole blown-up configuration space as the gauge invariant vector field
\begin{equation*}
(\mathrm{grad}\mathcal{L})^{\sigma}(B,r,\psi)=(\frac{1}{2}\ast F_{B^t}+r^2\rho^{-1}(\psi\psi^*)_0, \Lambda(B,r,\psi), D_B\psi-\Lambda(B,r,\psi)\psi),
\end{equation*}
see also Equation \ref{blowform}. Here $\Lambda(B,r,\psi)$ is the real number $\langle \psi, D_B\psi\rangle_{L^2}$ and it plays the same role as the quantity $\Lambda(\phi)$ of the previous section (with $D_B$ in place of $L$). The critical points of this vector field can be understood in the following terms.\begin{lem}\label{critblow}
Consider a configuration $\mathfrak{b}=(B,r,\psi)$.
\begin{itemize}
\item If $r\neq0$, then $\mathfrak{b}$ is a critical point of $(\mathrm{grad}\mathcal{L})^{\sigma}$ if and only if its blow down $\pi(\mathfrak{b})$ (which is an irreducible configuration) is a critical point of $\mathrm{grad}\mathcal{L}$.
\item If $r$ is zero, then $\mathfrak{b}$ is a critical point of $(\mathrm{grad}\mathcal{L})^{\sigma}$ if and only if its blow down $\pi(\mathfrak{b})$ (which is reducible configuration) is a critical point of $(\mathrm{grad}\mathcal{L})$ and $\psi$ is an eigenvector of $D_B$.
\end{itemize}
\end{lem}
This readily follows from the fact that the blow down is a diffeomorphism of the irreducible locus and the analogue of the description given in Exercise \ref{crit}.
\\
\par
We can then apply the construction of Morse homology with boundary described in the previous Section. To do this we need first to suitably perturb the Chern-Simons-Dirac functional $\mathcal{L}$ so that the critical points are non-degenerate (and in particular isolated). This means roughly speaking that the analogue of the Hessian is an isomorphism, but we will not enter the (quite delicate) technical details here. In light of Lemma \ref{critblow} we will require the following two conditions:
\begin{enumerate}
\item the critical points of $\mathrm{grad}\mathcal{L}$ are non-degenerate;
\item at the reducible critical points $(B,0)$ the spectrum of the (perturbed) Dirac operator $D_B$ is simple (i.e. the eigenspaces are one dimensional) and does not contain zero.
\end{enumerate}
The last condition implies that each eigenvalue of $D_B$ will give rise to a critical point in the boundary of $\mathcal{B}^{\sigma}(Y,\spin)$ (see again Exercise \ref{crit}).  Indeed, recall that we are restricting ourself to the unit sphere in the $L^2$ norm, so that each one dimensional eigenspace gives rise to a circle of critical points which are all identified under the action of the constant gauge transformations. The eigenvectors with positive eigenvalues correspond to stable critical points, while the ones with negative eigenvalues correspond to unstable critical points. Notice that the critical set in the blow-up is not compact (see Proposition \ref{compactness}) anymore because of the spectral theorem for Dirac operators (Lemma \ref{spectral}), so that there are infinitely many stable and unstable critical points for each reducible critical point.
\\
\par
Following the construction described in the previous section, we define the three chain complexes (which depend on the choice of metric and perturbation)
\begin{equation*}
(\check{C}_*(Y,\spin),\check{\partial}),\quad (\hat{C}_*(Y,\spin),\hat{\partial}),\quad (\bar{C}_*(Y,\spin),\bar{\partial})
\end{equation*}
whose underlying vector spaces (\ref{chain}) are generated by the critical points and whose differential (\ref{diff}) counts isolated trajectories. There is a complication regarding the gradings that we will discuss later in this section. The homologies
\begin{equation*}
\HMt_*(Y,\spin),\quad \HMf_*(Y),\quad \HMb_*(Y,\spin)
\end{equation*}
are the \textit{monopole Floer homology groups} associated to the given choice of metric and generic perturbation. The main invariance result is the following.
\begin{prop}\label{invariance}
The Floer homology groups are invariants of the pair $(Y,\spin)$, i.e. they are independent of the choice of metric and generic perturbation.
\end{prop}
\begin{rem}
In the case $\spin$ is non-torsion, we also require that no reducible solutions are present after perturbation, see Proposition \ref{redsol}.
\end{rem}

The naturality statement is interpreted in the same way as the discussion following Proposition \ref{disk}. In our case the continuation maps can be interpreted as maps induced by cobordisms, and we will discuss them thoroughly in the next section. We now focus on the basic example.

\begin{exm}\label{exS3}We compute the Floer homology groups of $S^3$ (see Proposition \ref{S3}). To do this, we suppose that it has the round metric. This has positive scalar curvature so Proposition \ref{weitz} tells us that before perturbing there are no irreducible solutions, while there is exactly one reducible solution thanks to Proposition \ref{redsol}. These two properties are stable under small perturbations so we can assume to have chosen a regular one (in the sense of Definition \ref{regularity}) with the same features. The critical points then form a doubly infinite tower lying over the reducible critical point, corresponding to the eigenvalues of the perturbed Dirac operator. We claim (see the upcoming Exercise \ref{traj} for the finite dimensional analogue) that the space of unparametrized trajectories between critical points corresponding to consecutive eigenvalues has dimension one. This implies that there are no differentials involved. In particular, $\HMt_*(Y,\spin)$, $\HMf_*(Y)$ and $\HMb_*(Y,\spin)$ are the (infinite dimensional) vector spaces generated by the stable, unstable, and both critical points. We will discuss gradings later.
\end{exm}
\begin{Exercise}\label{traj}
In the setting of Exercise \ref{crit}, show that the space of trajectories connecting the critical points corresponding to the eigenvalues $\lambda>\mu$ has dimension $2i-1$, where $i$ is the number of eigenvalues $\lambda\leq\nu<\mu$. Hint: consider a basis of eigenvectors $\{\phi_\nu\}$ labeled by the eigenvalue $\nu$. Then any trajectory entirely contained in the boundary of $(\mathbb{C}^n)^{\sigma}/S^1$ is the \textit{projectivization} of a trajectory in $\mathbb{C}^n$ of the form
\begin{equation*}
z(t)=\sum c_{\mu}e^{-\mu t}\phi_\mu
\end{equation*}
for some complex numbers $\{c_\mu\}$. When does this trajectory have the required limiting points?
\end{Exercise}

\vspace{0.5cm}

\textbf{Floer homology vs. Morse homology. }We briefly and informally discuss the main differences between Floer homology and Morse homology. First of all, of course we are working in an infinite dimensional setting, but Morse homology in infinite dimension was actually one of the first big achievements of Morse theory (\cite{Bott}). Indeed, one can consider the energy functional on loop spaces to study the existence of closed geodesics. The main difference with Floer homology is that in that case the Hessian at a critical point has finitely many negative eigenvalues, while in Floer homology it has infinitely many. This is a consequence of the spectral theorem for first order self adjoint operators in Lemma \ref{spectral}, and has two main consequences.
\\
\par
The first consequence is that there is not a well defined notion of index (hence of grading) anymore. Nevertheless there is a well defined notion of \textit{relative} grading, as we will discuss in detail.
\\
\par
The second consequence is more fundamental in nature. Unlike the Morse case, the ODE in the configuration space
\begin{equation}\label{flow}
\frac{d}{dt}(B,\Psi)=-\mathrm{grad}\mathcal{L}(B,\Psi)
\end{equation}
defining trajectories of the gradient flow \textit{cannot} be solved at a general point, not even for small time. To see this, we consider an analogous linear case in lower dimension, see Equation (\ref{dirac1}). Given a smooth complex valued function $u(\theta)$ on $S^1$, we want to find a solution $u(t,\theta)$ of
\begin{equation*}
\frac{d}{dt}u=-i\frac{d}{d\theta}u
\end{equation*}
with $u(0,\theta)=u(\theta)$. As $-id/d\theta$ is the Dirac operator on the circle, this problem has the same qualitative behavior as the linearization of (\ref{flow}). Given the Fourier decomposition $u(\theta)=\sum_{n\in\mathbb{Z}} a_n e^{n\theta}$, a solution (if it exists) has the form
\begin{equation*}
u(t,\theta)=\sum_{n\in\mathbb{Z}} a_ne^{nt} e^{n\theta}.
\end{equation*}
On the other hand, this sum is not well defined for $t\neq0$ if the coefficients $\{a_n\}$ do not decay sufficiently fast.
\par
One of Floer's key insights is that even though (\ref{flow}) does not make sense as an ODE, one could treat it as a PDE for which we have a much better grasp. Indeed, in our example case this PDE is just the Cauchy-Riemann equation on the annulus $\mathbb{R}\times S^1$. As we will see in the next section, the PDE associated to (\ref{flow}) are the original four dimensional Seiberg-Witten equations.
\\
\par
In classical Morse theory, in order to find trajectory space $M(a,b)$ between critical points $a$ and $b$, one intersects the respective stable and unstable manifolds. This cannot be carried over in our setting, as the latter do not exist. We discuss the main aspects of Floer's approach in the setting of usual Morse homology to avoid various complications that arise in the Seiberg-Witten one. The analytic foundations of this are provided by the classic work of Atiyah-Patodi-Singer (\cite{APS}). Suppose we are given two critical points $a,b$ of a Morse function $f$. Consider the space $P(a,b)$ consisting of maps
\begin{equation*}
\gamma:\mathbb{R}\rightarrow M
\end{equation*}
that converge sufficiently fast to $a$ at $-\infty$ and $b$ at $+\infty$. The tangent space $T_{\gamma}P(a,b)$ at a configuration $\gamma$ consists of sections of the bundle $\gamma^* TM$ over $\mathbb{R}$. The map
\begin{equation*}
\gamma\mapsto \frac{d}{dt}\gamma+\mathrm{grad}f(\gamma)
\end{equation*}
is then a section $s$ of the tangent bundle $TP(a,b)$ whose zero set is exactly the space of trajectories $M(a,b)$. In particular, if the section is transverse to the zero section, i.e. the linearization
\begin{equation}\label{MS}
ds: T_{\gamma}P(a,b)\rightarrow T_{\gamma}P(a,b)
\end{equation}
is surjective, then by the infinite dimensional inverse function theorem the set of trajectories $M(a,b)$ has a natural structure of smooth manifold. The linearization at $\gamma$ is the map
\begin{equation*}
\xi(t)\mapsto \frac{d}{dt}\xi(t)+ \mathrm{Hess}_{\gamma(t)}\xi(t),
\end{equation*}
where $\xi$ is a section of $\gamma^* TM$. In this case, Atiyah-Patodi-Singer theory tells us that, provided the critical points $a$ and $b$ are Morse, this operator is Fredholm and its index is exactly the difference of the indices of $a$ and $b$, as expected. This difference can be interpreted also as the \textit{spectral flow} of the family of self-adjoint operators $\mathrm{Hess}_{\gamma(t)}$, i.e. the number of eigenvalues (with sign) that goes from negative to positive (see Figure \ref{spflow}).

\begin{figure}  \centering
\def\svgwidth{\textwidth}
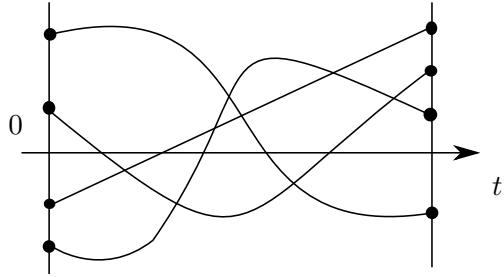
    \caption{A family of operators with spectral flow $1$. The lines connecting the dots indicate the evolution of the eigenvalues of the operators.}
    \label{spflow}

\end{figure} 

The key point of this reformulation is that it makes sense in infinite dimensions: the Morse-Smale condition is equivalent to the operator (\ref{MS}) being surjective at each point. The dimension of the space of trajectories (which is the analogue of the difference of the indices of the critical points) is the index of the operator, which can be interpreted as the spectral flow of a suitable family of self-adjoint operators.

\vspace{0.5cm}
\textbf{Relative gradings. }As mentioned before, even though the index of critical points is not well defined, we can still define relative gradings. Given two critical points $\mathfrak{a}$ and $\mathfrak{b}$ in $\mathcal{C}^{\sigma}(Y,\spin)$ we can define the quantity
\begin{equation}\label{grading}
\mathrm{gr}(\mathfrak{a},\mathfrak{b})\in\mathbb{Z}
\end{equation}
as the spectral flow of the Hessians of a path $\gamma$ connecting them. This is well defined because the spectral flow is invariant under homotopy and $\mathcal{C}^{\sigma}(Y,\spin)$ is simply connected. This number computes the expected dimension of the moduli space of trajectories connecting $\mathfrak{a}$ to $\mathfrak{b}$.
\par
On the other hand, we are interested in the moduli space of trajectories connecting the critical points \textit{up to gauge} $M([\mathfrak{a}],[\mathfrak{b}])$, where $[\mathfrak{a}]$ and $[\mathfrak{b}]$ denote the gauge equivalence classes of the critical points. The components of this space have different dimensions corresponding to the different lifts of $[\mathfrak{a}]$ and $[\mathfrak{b}]$. This is a manifestation of the fact that the moduli space of configurations $\mathcal{B}^{\sigma}(Y,\spin)$ has non-trivial fundamental group (see Proposition \ref{homotopytype}). In particular we can decompose the space of trajectories
\begin{equation}\label{hdec}
M([\mathfrak{a}],[\mathfrak{b}])=\bigcup_{z\in\pi_1([\mathfrak{a}],[\mathfrak{b}])}M_z([\mathfrak{a}],[\mathfrak{b}])
\end{equation}
as the union over the moduli spaces in a given relative homotopy class. Notice that in the definition of the Floer chain complexes we only consider the components of the moduli spaces which are zero dimensional.
\par
A consequence of this is that the relative grading (\ref{grading}) is not well defined up to gauge. Nevertheless, one can show its value is well defined in $\mathbb{Z}/ d\mathbb{Z}$, where $d$ is the positive generator of the image of
\begin{align*}
H^1(Y,\mathbb{Z})&\rightarrow \mathbb{Z}\\
x&\mapsto \langle c_1(\spin)\cup x,[Y]\rangle.
\end{align*}
We define this value to be the relative grading of the two critical points. As $c_1$ is even, $d$ is an even number. As a special case, when $\spin$ is torsion there is a well defined integral relative grading. We will discuss absolute gradings in more detail.

\begin{rem}\label{gradingchain}
One has to be a little careful when using this relative grading between critical points to define a relative grading on the chain complexes. This is because the grading (which is the dimension of the space of trajectories) is sometimes different from the difference of the indices. One needs to shift the gradings accordingly, as in Equation (\ref{chain}). 
\end{rem}

\begin{Exercise}
The fact that the relative grading is not well defined as an integer in infinite dimensions can also be interpreted as the fact that there are closed loops of self-adjoint operators with non trivial spectral flow. Find an explicit example. This shows that the space of such operators has non trivial fundamental group.
\end{Exercise}
\begin{rem}
It is important to notice that when $\spin$ is not torsion the Chern-Simons-Dirac functional is only circle valued. On the other hand the whole construction can be performed without need to use coefficients in the Novikov ring (as it is standard in those situations, see for example \cite{HutLee}). This is because it turns out that when the perturbation is regular there cannot be trajectories connecting a critical point to itself. This uses in an essential way the fact that the space of operators has non trivial topology as discussed above.
\end{rem}

\vspace{0.5cm}
\textbf{Duality and exact sequences. }
The monopole Floer groups are infinite dimensional analogues of the three Morse homology groups of the pair $(\mathcal{B}^{\sigma}(Y,\spin),\partial \mathcal{B}^{\sigma}(Y,\spin))$. From this point of view, equation \ref{LES} is nothing but the long exact sequence provided by Exercise \ref{finiteLES}.
\par
In the same spirit, the isomorphism of Proposition \ref{poincare} can be thought as a version of Poincar\'e -Lefschetz duality. In fact, given a spin$^c$ structure $\spin$ on $Y$ there is a natural spin$^c$ structure on $-Y$ (the manifold with orientation reversed) where the spinor bundle is unchanged and Clifford multiplication is given by $-\rho$. There is also a natural identification between $\mathcal{C}(Y,\spin)$ and $\mathcal{C}(-Y,\spin)$ and we have that the Chern-Simons-Dirac functionals are related by
\begin{equation*}
\mathcal{L}(-Y)=-\mathcal{L}(Y).
\end{equation*}
In particular, the critical points in the two cases are identified and the moduli spaces are obtained by reversing the time direction:
\begin{equation*}
M_z(Y;[\mathfrak{a}],[\mathfrak{b}])\cong M_z(-Y;[\mathfrak{b}],[\mathfrak{a}]).
\end{equation*}
The main difference is that the role of boundary stable and unstable critical points are reversed. The proof of Proposition \ref{poincare} is then easily understood in terms of finite dimensional Morse theory: changing the sign to $f$ provides natural identifications at the chain level
\begin{align*}
\check{C}_*(-f)&\cong \hat{C}^*(f)\\
\hat{C}_*(-f)&\cong \check{C}^*(f)\\
\bar{C}_*(-f)&\cong \bar{C}^*(f).
\end{align*}
On the right hand side we are considering the dual chain complexes, which compute the various cohomologies of the manifold with boundary $(M,\partial M)$. At the homology level this induces isomorphisms
\begin{align*}
H_*(M)&=H^*(M,\partial M)\\
H_*(M,\partial M)&=H^*(M)\\
H_*(\partial M)&= H^*(\partial M).
\end{align*}
\begin{Exercise}
What happens to the gradings under these identifications? Consider the finite dimensional case for simplicity.
\end{Exercise}

\vspace{1cm}
\section{Cobordisms}
In this section we discuss functoriality for monopole Floer homology. So far we have been considering the three dimensional version of the Seiberg-Witten equations. In fact, we have been telling the story in the opposite direction of how it has actually historically developed. The four dimensional equations were introduced in \cite{Wit} in order to define invariants of smooth four-manifolds. Up to the present day, these are the most powerful invariants available in the subject. A main drawback is that they are in general very hard to compute, as they involve counting the solutions of non-linear PDEs. In this sense a natural approach to this problem is to try to extend the theory to three-manifolds and cobordisms between them (in the spirit of topological quantum field theories): one can try to compute the invariants by cutting the four manifold into simpler pieces for which we know the computations, and then reconstruct the invariants of the global object using suitable pairing theorems. Monopole Floer homology was born with this purpose in mind (even though it turned out to be extremely useful in purely three dimensional problems as well). 
\\
\par
As said in the Introduction, we will assume that the reader is already familiar with the four-dimensional Seiberg-Witten equations. The more explicit description of a spin$^c$ structure introduced in Section \ref{SWeq} works as follows in dimension four. A spin$^c$ structure $\spin_X$ of a Riemannian four-manifold $X$ (possibly with boundary) consists of a rank $4$ hermitian bundle $S_X\rightarrow X$ together with a Clifford multiplication
\begin{equation*}
\rho: TX\rightarrow \mathrm{Hom}(S_X,S_X).
\end{equation*}
This is a bundle map such that at each point $x\in X$ for each oriented orthonormal frame $e_0,e_1,e_2,e_3$ of $T_xX$ there is an orthonormal basis of $S_x$ so that
\begin{equation*}
\rho(e_0)=\begin{bmatrix}
0 & -I_2 \\
I_2 & 0
\end{bmatrix},
\quad
\rho(e_i)=\begin{bmatrix}
0 & -\sigma_i^* \\
\sigma_i & 0
\end{bmatrix}
\quad (i=1,2,3).
\end{equation*}
These are $2\times 2$ block matrices with $2\times 2$ blocks. Here $I_2$ is the identity matrix while $\sigma_i$ are the Pauli matrices in Equation (\ref{pauli}). We can extend the Clifford multiplication to differential forms as in the three dimensional case. It is clear from this block description that $S_X$ splits as the direct sum of two rank two bundles $S^+$ and $S^-$ so that Clifford multiplication exchanges them. These can globally described as the $-1$ and $+1$ eigenspaces of the Clifford multiplication by the volume form. Furthermore Clifford multiplication provides an isometry
\begin{equation*}
\rho: \Lambda^+\rightarrow \mathfrak{su}(S^+)
\end{equation*}
where $\Lambda^+$ is the bundle of self-dual two-forms.
The Seiberg-Witten equations are then the equations for a pair $(A,\Phi)$ consisting of a spin$^c$ connection $A$ and a spinor $\Phi\in\Gamma(S^+)$ given by
\begin{align}\label{SW4}
\begin{split}
\frac{1}{2}\rho( F^+_{A^t})-(\Phi\Phi^*)_0 &=0\\
D_A^+\Phi&=0.
\end{split}
\end{align}
Again, here $A^t$ denotes the connection induced on the determinant line bundle $\Lambda^2 S^+$. Recall that the Dirac operator $D^+_A$ is the composition
\begin{equation*}
\Gamma(S^+)\stackrel{\nabla_A}{\longrightarrow}\Gamma(T^*X\otimes S^-)\stackrel{\rho}{\longrightarrow}\Gamma(S^-)
\end{equation*}
which is a first order elliptic operator whose index (if $X$ is closed) is
\begin{equation*}
(c_1(S^+)^2[X]-\sigma(X))/8
\end{equation*}
by the Atiyah-Singer index theorem.

\vspace{0.8cm}
We now discuss the relation between the four dimensional equation and the three dimensional ones by considering the case when $X$ is the product Riemannian manifold $I\times Y$. In this case Clifford multiplication by $\partial/\partial t$ (which can be used as the vector $e_0$ in the local picture above) provides an isometry between $S^+$ and $S^-$, which can hence be both identified with the spinor bundle $S$ of $Y$. A family of configurations $(B(t),\Psi(t))$ on $Y$ gives rise to a configuration $(A,\Phi)$ on $I\times Y$: the connection $A$ is defined at the level of covariant derivatives by
\begin{equation}\label{indconn}
\nabla_A=\frac{d}{dt}+\nabla_B.
\end{equation}
\begin{Exercise}
Check the identities
\begin{align*}
D_A^+&=\frac{d}{dt}+ D_B\\
F_{A^t}&=dt\wedge \left(\frac{d}{dt} B^t\right)+ F_{B^t}
\end{align*}
in the case $A$ arises from a family of connections $B(t)$ on $Y$.
\end{Exercise}
From these computations we have the following results.
\begin{prop}
Let $(A,\Phi)$ the four dimensional configuration on $I\times Y$ arising from a family $(B(t),\Psi(t))$ on $Y$. Then $(A,\Phi)$ solves (\ref{SW4}) if and only if $(B(t),\Psi(t))$ solves the three dimensional gradient flow equations (\ref{gradfloweq}).
\end{prop}
\begin{Exercise}
Prove the identity
\begin{equation*}
\ast_4 F_{A^t}= \ast\left( \frac{d}{dt} B^t\right)+ dt\wedge \ast F_{B^t}.
\end{equation*}
where we denote the four dimensional Hodge star by $\ast_4$ to avoid confusion. Use this identity to prove the Proposition above.
\end{Exercise}

This result provides us with the following interpretation:
\begin{align*}
\{\text{critical points of $\mathcal{L}$}\}&\rightarrow \{\text{translation invariant solutions of (\ref{SW4}) on $\mathbb{R}\times Y$}\}\\
\{\text{trajectories of $\mathrm{grad}\mathcal{L}$ from $[\mathfrak{a}]$ to $[\mathfrak{b}]$}\}&\rightarrow\{\text{solutions of (\ref{SW4}) that converge to $[\mathfrak{a}],[\mathfrak{b}]$ at $\pm\infty$}\}.
\end{align*}

It is important to remark that this correspondence is not bijective because not all the configurations on $I\times Y$ arise from families of configurations on $Y$. For a general connection on $I\times Y$, the covariant derivative in the $\partial/\partial t$ direction is not constant. On the other hand, a configuration $(A,\Phi)$ on $I\times Y$ always gives rise to a path of three dimensional configurations $(\check{A}(t),\check{\Phi}(t))$ by restricting to the slices. Nevertheless, the correspondence is a bijection after we quotient by action of the corresponding gauge groups, so we will not have to worry about this issue.

\vspace{0.5cm}
\textbf{Blowing up. }Suppose $X$ is compact. We can define the blown-up configuration space of a pair $(X,\spin)$ following the three dimensional case as the space of triples
\begin{equation*}
\mathcal{C}^{\sigma}(X,\spin_X)=\{(A,s,\phi)\}
\end{equation*}
where $s\in\mathbb{R}^{\geq0}$ and $\phi$ is a spinor with $\|\phi\|_{L^2}=1$. The Seiberg-Witten equations naturally extend to this space as the equations
\begin{align}\label{SW4bl}
\begin{split}
\frac{1}{2}\rho( F^+_{A^t})-s^2(\phi\phi^*)_0 &=0\\
D_A^+\phi&=0.
\end{split}
\end{align}
There is a subtlety when comparing this to the three dimensional case on a product $I\times Y$. A configuration $(A,s,\phi)$ on the blown up configurations $I\times Y$ gives rise to a path of configurations on $Y$ given by
\begin{equation*}
(\check{A}(t), s\|\check{\phi}(t)\|_{L^2(Y)} , \check{\phi}(t)/\|\check{\phi}(t)\|_{L^2(Y)}),
\end{equation*}
where the check denotes the restriction to $\{t\}\times Y$. Of course, this makes sense provided that $\|\check{\phi}(t)\|_{L^2(Y)}$ is non-zero for each $t\in I$. This is generally false, but following unique continuation result tells us that when restricting our attention to solutions of the blown-up Seiberg-Witten equations the latter condition is actually always satisfied. This is the analogue for the Dirac operator of the unique continuation property for the usual $\bar{\partial}$ operator.
\begin{prop}\label{soleqbl}
Let $\phi$ be a solution of $D_A^+\phi=0$ on $I \times Y$. If the restriction of $\phi$ to $\{t\}\times Y$ is zero for some $t\in I$, then $\phi$ is identically zero. In particular there is a one to one correspondence between (gauge equivalence classes of) integral paths of the vector field $(\mathrm{grad}\mathcal{L})^{\sigma}$ in $\mathcal{C}^{\sigma}(Y,\spin)$ and solutions of (\ref{SW4bl}).
\end{prop}

\vspace{0.5cm}
\textbf{The map induced by a cobordism. }We now describe how a spin$^c$ cobordism $(W,\spin_W)$ between two spin$^c$ three manifolds $(Y_-,\spin_-)$ and $(Y_+,\spin_+)$ induces a map between the monopole Floer homology groups. Suppose that $W$ has a Riemannian metric which is cylindrical near the boundary, i.e. isometric to a product $I\times (\bar{Y}_-\amalg Y_+)$. We can construct the Riemannian four manifold with cylindrical ends $W^*$ by attaching the two half infinite ends $(-\infty,0]\times Y_-$ and $[0,\infty)\times Y_+$. The spin$^c$ structure on $W$ naturally extends to $W^*$.
\par
Given two critical points $[\mathfrak{b}_{\pm}]$ for the (perturbed) Chern-Simons-Dirac functional on $Y_{\pm}$, we can consider the space
\begin{equation*}
M(W^*, \spin_W, [\mathfrak{b}_{-}],[\mathfrak{b}_{+}])
\end{equation*} 
of solutions to the equations (\ref{SW4bl}) in the blow-up that converge to $[\mathfrak{b}_{-}]$ respectively on the incoming/outgoing cylindrical end. This makes sense because by the second part of Proposition \ref{soleqbl} we can interpret the restriction of the solutions to the cylinders $(-\infty,0]\times Y_-$ and $[0,\infty)\times Y_+$ as half trajectories for the vector fields $(\mathrm{grad}\mathcal{L})^{\sigma}$ in the blow up.
\begin{rem}
As we are dealing with a non compact manifold, the definition of blow up has to be slightly changed, see the Remark in Section $6.1$ of \cite{KM}. This will not affect the rest of our discussion.
\end{rem}
These equations can be perturbed in the interior of the compact cobordism $X$ so that all these moduli spaces are transversely cut out. Notice that the space of configurations converging to $[\mathfrak{b}_{\pm}]$ decomposes according to the relative homotopy class. This is analogous to the decomposition (\ref{hdec}) according to the relative homotopy class of the path. In particular we can write
\begin{equation*}
M(W^*, \spin_W, [\mathfrak{b}_{-}],[\mathfrak{b}_{+}])=\bigcup_{z} M_z(W^*, \spin_W, [\mathfrak{b}_{-}],[\mathfrak{b}_{+}])
\end{equation*} 
The expected dimension of the moduli spaces changes from component to component, and we define this to be the \textit{relative grading}. To define the maps induced by the cobordism, we will count solutions in zero dimensional moduli spaces (recall that solutions are not translations invariant anymore). In particular one can define the linear map
\begin{equation*}
m^o_o: C^o(Y_-)\rightarrow C^o(Y_+)
\end{equation*}
by counting the (finitely many) solutions on the cobordism connecting irreducible critical points. These are necessarily irreducible. Similarly by counting irreducible solutions we obtain the maps $m^o_s, m^u_s$ and $m^u_o$, while reducible solutions give rise to $\bar{m}^s_s, \bar{m}^s_u, \bar{m}^u_s, \bar{m}^u_u$. We then package these into the maps
\begin{align}\label{indmap}
\begin{split}
\check{m}:\check{C}_*(Y_-)\rightarrow \check{C}_*(Y_+)\\
\hat{m}:\hat{C}_*(Y_-)\rightarrow \hat{C}_*(Y_+)\\
\bar{m}:\bar{C}_*(Y_-)\rightarrow \bar{C}_*(Y_+)
\end{split}
\end{align}
given in components by
\begin{align*}
\check{m}&=
\begin{bmatrix}
m^o_o & m^u_o\bar{\partial}^s_u(Y_-)+\partial^u_o(Y_+)\bar{m}^s_u\\
m^o_s & \bar{m}^s_s+ m^u_s\bar{\partial}^s_u(Y_-)+ \partial^u_s(Y_+)\bar{m}^s_u
\end{bmatrix}\\
\hat{m}&=\begin{bmatrix}
m^o_o & m^u_o\\
\bar{m}^s_u\partial^o_s(Y_-)+\bar{\partial}^s_u(Y_+)m^o_s &  \bar{m}^u_u+\bar{m}^s_s\partial^u_s(Y_-)+\bar{m}^s_u(Y_=)m^u_s
\end{bmatrix}\\
\bar{m}&=
\begin{bmatrix}
\bar{m}^s_s & \bar{m}^u_s\\
\bar{m}^s_u & \bar{m}^u_u
\end{bmatrix}.
\end{align*}
Here, for example, $\bar{\partial}^s_u(Y_-)$ counts reducible trajectories on $Y_-$ from stable to unstable critical points.
\begin{prop}
The maps $\check{m}$, $\hat{m}$ and $\bar{m}$ are chain maps. The map induced in homology is independent of the choice of the metric and perturbation on $W$.
\end{prop}
We say that the induced maps are the maps induced by the cobordism, and denote them by
\begin{equation*}
\HMt_*(W,\spin_W),\quad \HMf_*(W,\spin_W),\quad \HMb_*(W,\spin_W).
\end{equation*}
\begin{proof}
We only provide a sketch of the proof and refer to Chapter $26$ of \cite{KM} for details. The first statement follows from understanding the compactification of the one dimensional moduli spaces. For example if both endpoints are irreducible critical points, a sequence of solutions can converge (in a suitable sense) to one of the following objects:
\begin{itemize}
\item a pair $([\check{\gamma}_-],[\gamma])$ consisting of an unparametrized trajectory $[\check{\gamma}_-]$ on $\mathbb{R}\times Y_-$ converging to $[\mathfrak{b}_-]$ at $-\infty$ and to another irreducible configuration $[\mathfrak{b}]$ at $+\infty$ and a solution $[\gamma]$ on $W^*$ converging to $[\mathfrak{b}]$ on the incoming end and to $[\mathfrak{b}_+]$ on the outgoing end.
\item a pair $([\gamma],[\check{\gamma}_+])$ consisting of a solution $[\gamma]$ on $W^*$ converging to $[\mathfrak{b}_-]$ on the incoming end and to an irreducible critical point $[\mathfrak{b}]$ on the outgoing end and an unparametrized trajectory $[\check{\gamma}_+]$ on $\mathbb{R}\times Y_+$ converging to $[\mathfrak{b}]$ at $-\infty$ and to another irreducible configuration $[\mathfrak{b}_+]$ at $+\infty$.
\item a triple $([\check{\gamma}_-],[\gamma],[\check{\gamma}_+])$ where the solutions are as above and $[\check{\gamma}_-]$  converges to a stable critical point at $+\infty$ and $[\check{\gamma}_+]$ converges to an unstable critical point at $-\infty$.
\end{itemize}
One shows that the compactified moduli spaces have the structure of a compact topological one dimensional manifold, and the description of the boundary points above gives rise to the identity
\begin{equation*}
m^o_o\partial^o_o(Y_-)+\partial^o_o(Y_+)m^o_o+ \partial^u_o\bar{m}^s_u\partial^o_s=0.
\end{equation*}
This is the analogue of the identity (\ref{osuo}). By checking the various possibilities one then checks that the maps are actually chain maps.
\par
Suppose we are given two different choices of perturbations $\mathfrak{p}_0$ and $\mathfrak{p}_1$, and denote by $\check{m}_0$ and $\check{m}_1$ the two chain maps they induce. One can choose a one parameter family of perturbations $\mathfrak{p}_t$ for $t\in[0,1]$. We can consider the moduli spaces parametrized by this family, i.e. the union
\begin{equation*}
\bigcup_{t\in[0,1]} M_z(X^*, \spin_X, [\mathfrak{b}_{-}],[\mathfrak{b}_{+}])_{\mathfrak{p}_t}.
\end{equation*}
For a generic choice of the family, these are all transversely cut out. By looking at zero dimensional moduli spaces one can construct a linear map
\begin{equation*}
\check{H}: \check{C}_*(Y_-)\rightarrow \check{C}_*(Y_+)
\end{equation*}
and looking at the ends of compactified one dimensional moduli spaces one obtains the identity
\begin{equation*}
\check{\partial}(Y_+)\circ\check{H}+\check{H}\circ\check{\partial}(Y_-)=\check{m}_0+\check{m}_1.
\end{equation*}
Hence the chain maps $\check{m}_0$ and $\check{m}_1$ are chain homotopic.
\end{proof}
\begin{rem}\label{invrem}
The invariance result of Proposition \ref{invariance} follows by considering the map induced by the cobordism $I\times Y$ equipped with a pair of metric and perturbation interpolating the given two on $Y$ at the ends. This is the equivalent in our setting of the continuation map in usual Morse theory (see equation (\ref{continuation})). The composition property for these maps is a special case of the functoriality property in Proposition \ref{functoriality}, which we discuss more thoroughly in the next Section. 
\end{rem}

\vspace{1cm}
\section{The module structure, completions and functoriality}
In this section we introduce some additional aspects related to the maps induced by cobordisms.

\vspace{0.3cm}

\textbf{Cap products.}
The Floer homology groups have a natural module structure over the (singular) cohomology ring of $\mathcal{B}^{\sigma}(Y,\spin)$. These are the infinite dimensional analogues of the cap products induced on the homology of a smooth manifold
\begin{equation*}
H^k(M)\times H_m(M)\rightarrow H_{m-k}(M).
\end{equation*}
We first identify the cohomology of the moduli space of configurations.

\begin{prop}\label{homotopytype}
The blown-up moduli space of configurations $\mathcal{B}^{\sigma}(Y,\spin)$ has the homotopy type of $\mathbb{T}\times \mathbb{C}P^{\infty}$ where $\mathbb{T}$ is a torus of dimension $b_1(Y)$. In particular its cohomology is the ring (\ref{ring}).
\end{prop}
\begin{proof}
Fix a reducible base configuration $\mathfrak{b}=(B_0,0,0)$. We define the Coulomb slice $\mathcal{S}_{\mathfrak{b}}$ through $\mathfrak{b}$ as the space of configurations $(B_0+b, r, \psi)$ so that 
\begin{equation*}
d^*b=0.
\end{equation*}
Given any configuration $(B_0+b, r, \psi)$ there is exactly one gauge transformation of the form $u=e^{\xi}$ with $\int_Y\xi=0$ so that the action lies on the Coulomb slice. This boils down to solving the equation
\begin{equation*}
\Delta \xi= d^*b
\end{equation*}
which has exactly one solution with $\int_Y\xi=0$ by standard Hodge theory. In particular we can identify $\mathcal{B}^{\sigma}(Y,\spin)$ as the quotient of $\mathcal{S}_{\mathfrak{b}}$ by the action of the group
\begin{equation*}
\mathcal{G}(Y,\spin)/\{e^{\xi}\lvert \int_Y\xi=0\}
\end{equation*}
This is identified (non-canonically) with $S^1\times H^1(Y;\mathbb{Z})$, see Exercise \ref{cohom}. As $\mathcal{S}_{\mathfrak{b}}$ is a contractible space, the result follows.
\end{proof}

We only provide a heuristic definition of the cap products. First recall the case of a closed finite dimensional manifold $M$. Suppose one can represent the Poincar\'e dual of a class $\alpha\in H^k(M)$ by a closed smooth submanifold $P_{\alpha}$ (which has henceforth dimension $m-k$). Consider two critical points $a$ and $b$ whose indices differ by $k$. The space of trajectories $M(a,b)$ is a smooth manifold of dimension $k$, and for a generic choice of the metric and Morse function we can assume that it intersects $P_{\alpha}$ transversely in a finite number of points. More precisely one should consider the closure of the space of trajectories, but the intersections will all be involving $M(a,b)$ itself because of transversality. Then the map
\begin{align*}
C_*&\rightarrow C_{*-k}\\
a&\mapsto \sum \#(M(a,b)\cap P_{\alpha})b
\end{align*}
is a chain map and the induced map is the cap product with $\alpha$. Analogously, we can fix a cocycle representative of $\alpha$ and evaluate it on $M(a,b)$.
\par
The same construction can be performed for the maps induced by a cobordism $W$. We can define for example the linear map
\begin{equation*}
m^o_o(U^k): C^o\rightarrow C^o
\end{equation*}
counting the intersections of $2k$ dimensional moduli spaces connecting irreducible critical points with a geometric representative of the Poincar\'e dual of $U^k$ in the singular cohomology ring of $\mathcal{B}^{\sigma}(Y,\spin)$.
The latter in this case is easy to describe, as the spinor component has the homotopy type of $\mathbb{C}P^{\infty}$. In particular, for $N$ big enough, the dual of $U^k$ is represented by a $N-k$-dimensional complex projective subspace and this description is compatible with stabilizations. Using formulas analogous to those defining the chain maps in equation (\ref{indmap}), we obtain the chain maps $\check{m}(U^k),\hat{m}(U^k)$ and $\bar{m}(U^k)$ defining the cap product structure.

\begin{exm}\label{modS3}
We focus on the case of $S^3$ described in Proposition \ref{S3}. We have already established a group isomorphism, so we need to determine the action of $U$. The union of the flow lines between two critical points corresponding to consecutive eigenvalues is a complex projective line with two points removed  (see the discussion in Exercise \ref{traj}). In particular (under genericity assumptions) it intersects the Poincar\'e dual of $U$ in exactly one point, and the result follows.
\end{exm}

\vspace{0.5cm}

\textbf{Completions and functoriality.} In order to discuss compositions (and hence functoriality) we need to introduce suitable completions of the monopole Floer groups. Indeed, in the previous section we only have dealt with the map induced by a single spin$^c$ structure on the cobordism $W$. In general, there might be infinitely spin$^c$ structures on $W$ for which the induced map does not vanish (for an explicit example, see the proof of Proposition \ref{fro}). In particular, the sum over all the spin$^c$ structures of the induced maps (see equation (\ref{totalmap})) is in general not well-defined. 
\par
The Floer homology groups can be endowed with a natural decreasing filtration induced by the relative grading as follows. If $\spin$ is not torsion, the filtration is trivial. If $\spin$ is torsion, there is a relative $\mathbb{Z}$-grading which can be lifted to a non-canonical absolute grading by choosing any critical point. The filtration is the one induced by the grading. We define the \textit{completed monopole Floer homology groups}, denoted by
\begin{equation*}
\HMt_{\bullet}(Y,\spin),\quad\HMf_{\bullet}(Y,\spin),\quad\HMb_{\bullet}(Y,\spin),
\end{equation*}
as the completion of the monopole Floer groups with respect to this filtration.
\begin{exm}
Following Proposition \ref{S3}, we have the identifications
\begin{align*}
\HMt_{\bullet}(S^3)&\cong \ztwo[U^{-1},U]]/\ztwo[[U]]\\
\HMf_{\bullet}(S^3)&\cong \ztwo[[U]]\langle-1\rangle\\
\HMb_{\bullet}(S^3)&\cong \ztwo[U^{-1},U]].
\end{align*}
The last group is the group of Laurent power series. Notice that the  completion does not affect the \textit{to} groups, as it vanishes in degrees low enough.
\end{exm}
The total Floer groups (\ref{total}) are obtained by taking the direct sum over all spin$^c$ structures. It is then a consequence of the compactness properties of the Seiberg-Witten equations that the total map (\ref{totalmap}) is well defined.
\\
\par
The proof of the functoriality property in Proposition \ref{functoriality}, which underlies the invariance of the Floer groups (see Remark \ref{invrem}) and the fact that the cap product defined above is actually a module structure, follows from a neck stretching argument (see Figure \ref{neck}). We can form the cobordism $W_T$ obtained from the composition $W_2\circ W_1$ by inserting a cylinder $[-T,T]\times Y_1$. Of course, all these cobordisms are diffeomorphic, but the metric is different. In particular, a family of solutions on the manifold with cylindrical ends $(W_T)^*$ for $T$ going to $+\infty$ converges in a suitable sense to a concatenations of solutions on $(W_1)^*$, $(W_2)^*$ and trajectories on $Y_0,Y_1$ and $Y_2$. This allows to identify the map induced by the composition in terms of the maps induced by the two cobordisms, see Chapter $26$ in \cite{KM} for details.

\begin{figure}
  \centering
\def\svgwidth{\textwidth}
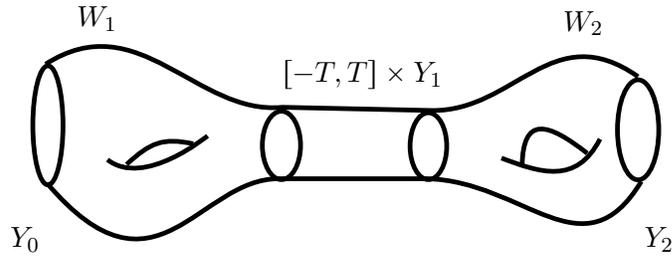
    \caption{The cobordism $W_T$ obtained by inserting a long neck $[-T,T]\times Y_1$.}
    \label{neck}
\end{figure}

\vspace{1cm}
\section{Absolute gradings and the Fr\o yshov invariant}

\textbf{Absolute rational gradings. }Suppose the spin$^c$ structure is torsion. We have seen that under this assumption a relative $\mathbb{Z}$-grading can be defined on the Floer homology groups. We show how this can be refined to an absolute rational grading which is nicely compatible with the maps induced by cobordisms.
\\
\par
We first review a little algebraic topology. For a manifold with boundary $W$ there is a non-degenerate pairing
\begin{equation*}
H^2(W,\partial W;\mathbb{Q})\times H^2(W;\mathbb{Q})\rightarrow H^4(W,\partial W;\mathbb{Q})\rightarrow \mathbb{Q}
\end{equation*}
given by composing the cup product with the evaluation on the relative fundamental class. This induces a non-degenerate quadratic form on 
\begin{equation*}
I(W)=\mathrm{im}\left\{H^2(W,\partial W;\mathbb{Q})\rightarrow H^2(W;\mathbb{Q})\right\},
\end{equation*}
the image of the restriction map. We denote the signature of this quadratic form by $\sigma(W)$. Furthermore for a cobordism $W$ between $Y_-$ and $Y_+$ we define the characteristic number 
\begin{equation*}
\iota(W)=\frac{1}{2}(\chi(W)+\sigma(W)+b_1(Y_-)-b_1(Y)).
\end{equation*}
\begin{Exercise}
Check that $\iota(W)$ is an integer and it is additive under composition.
\end{Exercise}

Suppose we are given a pair $(Y,\spin)$ with $\spin$ torsion, and choose a spin$^c$ four manifold $(X,\spin_X)$ whose boundary is $(Y,\spin)$. Consider the class $c_1(S^+)$ which lies in $H^2(X,\mathbb{Z})$. As $\spin$ is torsion, the image of this class in $H^2(X,\mathbb{Q})$ is in the vector space $I(X)$ defined above, so that we can define the self-intersection number $c_1(S^+)^2\in\mathbb{Q}$.
\begin{Exercise}
Consider the tubular neighborhood $N(-p)$ of a sphere of self-intersection $-p$. The boundary is naturally identified with $L(p,1)$. Compute the self intersection numbers for the classes in $H^2(N(-p);\mathbb{Z})$.
\end{Exercise}

We can remove  a ball from $X$ and consider it as a spin$^c$ cobordism $W$ from $S^3$ to $Y$. On $S^3$ (endowed with the standard metric and a small perturbation) we have the stable critical point $[\mathfrak{a}_0]$ corresponding to the lowest degree generator of $\HMt_*(S^3)$, see Example \ref{exS3}.
\begin{defn}
For a critical point $[\mathfrak{a}]$ on $(Y,\spin)$ with $\spin$ torsion we define its absolute grading to be
\begin{equation*}
\mathrm{gr}^{\mathbb{Q}}=-\mathrm{gr}_z([\mathfrak{a}_0],W,[\mathfrak{a}])+\frac{1}{4} c_1(S^+)^2-\iota(W)-\frac{1}{4}\sigma(W)
\end{equation*}
for any choice of homotopy class $z$ on $W$, where $\mathrm{gr}_z([\mathfrak{a}_0],W,[\mathfrak{a}])$ is the formal dimension of the moduli space of solutions connecting $[\mathfrak{a}_0]$ to $[\mathfrak{a}]$ in the homotopy class $z$.
\end{defn}
\begin{rem}
This definition is slightly different from the convention in Heegaard Floer homology (see \cite{OSd}). In particular, that is obtained from ours by subtracting $b_1(Y)/2$.
\end{rem}

To see that the absolute grading is well defined, we notice that the quantities involved in the expression are all additive under composition, so that we can reduce ourself to show that the right hand side is zero on a closed four manifold $X$. This boils down to the usual formula for the dimension of the Seiberg-Witten moduli spaces
\begin{equation*}
d=\frac{1}{4}\left(c_1(S^+)^2-2\chi(X)-3\sigma(X)\right).
\end{equation*}
\begin{exm}It is straightforward from the definition (by taking $X$ to be $B^4$) that the lowest degree generator of $\HMt_*(S^3)$ has degree zero.\end{exm}

The following properties will be useful when discussing the Fr\o yshov invariant.
\begin{prop}\label{gradingprop}
The absolute gradings satisfy the following properties.
\begin{enumerate}
\item If $Y$ is a homology sphere and $[\mathfrak{a}]$ is a stable reducible critical point then $\mathrm{gr}^{\mathbb{Q}}([\mathfrak{a}])$ is an even integer.
\item If $\spin_0$ and $\spin_1$ are torsion, the associated map
\begin{equation*}
\HMt_*(W,\spin):\HMt_*(Y_0,\spin_0)\rightarrow \HMt_*(Y_1,\spin_1)
\end{equation*}
has degree given by the rational number
\begin{equation*}
\frac{1}{4}c_1(S^+)^2-\iota(W)-\frac{1}{4}\sigma(W).
\end{equation*}
\item When $\spin$ is torsion, the duality isomorphism
\begin{equation*}
\check{\omega}:\HMt_*(-Y)\rightarrow \HMf^*(Y)
\end{equation*}
maps elements of grading $j$ to elements of grading $-1-b_1(Y)-j$.
\end{enumerate}
\end{prop}
\begin{proof}
We sketch the proof of part $(1)$ and $(2)$. We discuss the various terms appearing in the definition of the grading. Because the boundary is a homology sphere, the intersection form of $W$ is unimodular, hence the square of any characteristic vector is the signature modulo $8$ (see for example \cite{Kir}). Thus
\begin{equation*}
\frac{1}{4}c_1(S^+)^2-\frac{1}{4}\sigma(W)
\end{equation*}
is an even integer. The quantity $\mathrm{gr}_z([\mathfrak{a}_0],W,[\mathfrak{a}])$ can be computed as the index of the linearization of the Seiberg-Witten equations at a configurations connecting $[\mathfrak{a}_0]$ to $[\mathfrak{a}]$. We can choose this to be reducible by hypothesis. In this case linearization has two parts, one given by a Dirac operator (which is a complex operator, hence it has even real index) and an operator obtained by perturbing
\begin{equation*}
d^*+d^+: \Omega^1(W^*)\rightarrow \Omega^0(W^*)\oplus \Omega^+(W^*),
\end{equation*}
which has index $-\iota(W)$.\end{proof}

\begin{Exercise}
When using this grading to define a grading on the actual chain complexes, a little extra care has to be taken, see Remark \ref{gradingchain}. In particular, show that the highest grading in $\HMf_*(S^3)$ is $-1$.
\end{Exercise}

\vspace{0.5cm}

\textbf{The Fr\o yshov invariant. }In order to make sense of Definition \ref{froyshov}, the key result needed is Proposition \ref{rational}. As it only involves reducible solutions, its proof follows from Example \ref{exS3} and \ref{modS3}. They key property satisfied by the Fr\o yshov invariant is the following.
\begin{prop}[\cite{Fro}]\label{fro}
Suppose $Y_0$ and $Y_1$ are two homology spheres and $W$ is a cobordism between them with negative definite intersection form $Q$. Let 
\begin{equation*}
\rho(Q)=\frac{1}{8}(\mathrm{rank}(Q)-\mathrm{inf}_c|Q(c)|)
\end{equation*}
where the infimum is taken over all characteristic vectors. Then $h(Y_0)\geq h(Y_1)+\rho(Q)$.
\end{prop}

\begin{Exercise}\label{E8}
The Poincar\'e homology sphere $P$ (with the correct choice of orientation) bounds the negative definite $-E_8$ plumbing. Show that this implies that $h(P)\leq -1$. One can show that $h(P)$ is in fact $-1$.
\end{Exercise}

It is a result of Elkies (\cite{Elk}) that for any quadratic form $Q$ the quantity $\rho(Q)$ is always non-negative and it is zero if and only if the form $Q$ is standard. Hence the groundbreaking result of Donaldson (\cite{Don4}) follows by considering a closed four manifold as a cobordism form $S^3$ to $S^3$.
\begin{cor}
Suppose a closed four manifold $X$ has negative definite intersection form $Q_X$. Then $Q_X$ is standard.
\end{cor}

Another important corollary is the following.
\begin{cor}\label{froy}
The Fr\o yshov invariant is invariant under homology cobordisms. Furthermore $h(-Y)=-h(Y)$.
\end{cor}
\begin{proof}
A homology cobordism $W$ gives us an inequality of the form $h(Y_0)\geq h(Y_1)$. The reversed cobordism with the reversed orientation gives us the opposite inequality, so the result follows.
\end{proof}
\begin{Exercise}
Prove the second part of the corollary using Poincar\'e duality and Proposition \ref{gradingprop}.
\end{Exercise}

\begin{proof}[Proof of Proposition \ref{fro}]
First of all we can perform surgeries on the cobordism $W$ without affecting the intersection form so that $b_1=0$, so we assume the latter holds. Our claim is that for every spin$^c$ structure $\spin$ on $W$, the map
\begin{equation*}
\HMb_*(W,\spin):\HMb_*(Y_0)\rightarrow \HMb_*(Y_1)
\end{equation*}
induces an isomorphism. Notice that by Proposition \ref{rational} and bullet $(1)$ of Proposition \ref{gradingprop} these groups consist of a summand $\mathbb{F}$ in each even degree. Furthermore the map $\HMb_*(W,\spin)$ has degree $(b_2(W)-c_1(S^+)^2)/4$ from bullet $(2)$ of Proposition \ref{gradingprop} and the fact that the cobordism is negative definite. From this the result follows (by taking the infimum over all characteristic vectors) because there is the commutative diagram of $\mathbb{F}[U]$-modules

\begin{center}
\begin{tikzpicture}
  \matrix (m) [matrix of math nodes,row sep=2em,column sep=5em,minimum width=2em]
  {
  \HMb_{\bullet}(Y_0,\spin_0) & \HMb_{\bullet}(Y_1,\spin_1)\\
    \HMt_{\bullet}(Y_0,\spin_0) & \HMt_{\bullet}(Y_1,\spin_1)\\
   };
  \path[-stealth]
  
  (m-1-1) edge node [above] {$\HMb_{\bullet}(W,\spin)$}  (m-1-2)
    edge node [left] {$i_*$} (m-2-1)
 (m-2-1) edge node [above] {$\HMt_{\bullet}(W,\spin)$} (m-2-2)
 (m-1-2) edge node [left] {$i_*$} (m-2-2)
    ;
   
\end{tikzpicture}
\end{center}
and the bottom element of the tower in \textit{HM-to} (whose grading determines the Fr\o yshov invariant) is the element of minimum degree in the image of $i_*$.
\par
The claim follows by directly inspecting the \textit{reducible} moduli spaces involved. This is essentially the same as Exercise \ref{traj} with a little extra input of Hodge theory on manifolds with cylindrical ends (\cite{APS}). The latter behaves qualitatively like that on closed manifold because the ends are modeled on homology spheres. Indeed, there is exactly one solution to the unperturbed equations $F_{A^t}^+=0$ as $b_1(W)=0$, and this is regular because $b_+(W)=0$. The solutions to the Seiberg-Witten equations in the blow-up correspond then as in Exercise \ref{traj} to the projectivization of the elements in the kernel of the Dirac operator $D^+_{A_0}$ which have the right asymptotics at both ends. In particular, when the moduli space is zero dimensional (i.e. when the difference in grading is exactly the degree of the map), it consists of a single point, and the result follows.
\end{proof}

\vspace{1cm}
\section{$\mathrm{Pin}(2)$-monopole Floer homology} 
In this section we discuss how Manolescu's recent $\Pin$-Seiberg-Witten Floer homology \cite{Man2} fits in the framework we have developed so far. We refer the reader to the nice surveys \cite{Man3} and \cite{Man4} for a thorough introduction to his approach, which is based on Conley index theory and finite dimensional approximations. The content of the next two sections come from Chapter $4$ of \cite{Lin}.

\vspace{0.3cm}
\textbf{Self-conjugate spin$^c$ structures. }There is a natural action $\jmath$ on the set of spin$^c$ structures $\mathrm{Spin}^c(Y)$ given by complex conjugation. We denote the orbits by $[\spin]$, and call the fixed points of this action \textit{self-conjugate spin$^c$-structures}.\par
Self-conjugate spin$^c$-structures are tightly related to spin structures. Recall that a spin structure can be thought as a lift of the $SO(3)$ frame bundle to a $\mathrm{Spin}(3)$ bundle. From the observation that $\mathrm{Spin}(3)$ can be identified with $\mathrm{SU}(2)$ and the spinor representation is just the usual action on $\mathbb{C}^2=\mathbb{H}$, we can recover the definition we have given of spin$^c$ structure by taking the associated bundle. The multiplication by $j$ on $\mathbb{H}$ from the right defines a quaternionic structure on the Clifford bundle $S$, i.e. a complex anti-linear automorphism such that $j^2=-\mathrm{Id}$. In particular, if $\spin$ is induced by a spin structure the multiplication by $j$ induces an isomorphism between $\spin$ and $\bar{\spin}$. We have just shown that there is a map
\begin{equation*}
\{\text{spin structures on $Y$}\}\rightarrow \{\text{self-conjugate spin$^c$ structures on $Y$}\}.
\end{equation*}
\begin{prop}\label{spincself}
This map is surjective and $2^{b_1(Y)}$ to $1$ (the first Betti number is taken with $\mathbb{Z}$-coefficients).
\end{prop}

\begin{proof}
Recall that spin structures are classified by $H^1(Y;\mathbb{Z}/2\mathbb{Z})$ (see for example \cite{GS}), while spin$^c$ structures are classified by $H^2(Y;\mathbb{Z})$. In the description of Lemma \ref{spincclass}, if we fix the base spin$^c$ structure to be self-conjugate, the conjugation action on the set of spin$^c$ structures $\mathrm{Spin}^c(Y)$ corresponds to the conjugation action of the set of complex line bundles on $Y$. As
\begin{equation*}
c_1(\bar{L})=-c_1({L}),
\end{equation*}
the self-conjugate spin$^c$ structures are identified with the $2$-torsion of $H^2(Y;\mathbb{Z})$. Using the Bockstein exact sequence
\begin{equation*}
\cdots {\longrightarrow} H^1(Y;\mathbb{Z})\stackrel{\cdot2}{\longrightarrow}H^1(Y;\mathbb{Z})\longrightarrow H^1(Y;\mathbb{Z}/2\mathbb{Z}){\longrightarrow} H^2(Y;\mathbb{Z})\stackrel{\cdot2}{\longrightarrow} H^2(Y;\mathbb{Z})\longrightarrow\cdots
\end{equation*}
the result follows.
\end{proof}
For example, on $S^2\times S^1$ the two spin structures both induce the same spin$^c$ structure (the only torsion one), while on $\mathbb{R}P^3$ the two spin structures induce different spin$^c$ structures.

\vspace{0.5cm}
\textbf{The formal picture. }To each self-conjugate spin$^c$ structure $\spin$ we associate the $\Pin$-\textit{monopole Floer homology groups}
\begin{equation}\label{pin2}
\HSt_*(Y,\spin),\qquad \HSf_*(Y,\spin),\qquad \HSb_*(Y,\spin).
\end{equation}
As $\spin$ is torsion (see the proof of Proposition \ref{spincself}), these groups carry as in the usual case a relative $\mathbb{Z}$-grading and an absolute $\mathbb{Q}$-grading. They also carry a structure of graded module over the ring
\begin{equation*}
\Rin= \ztwo[V][Q]/(Q^3)
\end{equation*}
where the actions of $V$ and $Q$ have degree respectively $-4$ and $-1$. These three groups should be thought as the homology groups of the quotient of $\mathcal{B}^{\sigma}(Y,\spin)$ by a natural fixed-point free involution $\jmath$. These three groups fit the long exact sequence analogue to (\ref{LES}), and they satisfy the analogue of Poincar\'e-Lefschetz duality (Proposition \ref{poincare}) with the respective cohomological version. Furthermore, they fit in the Gysin exact sequence
\begin{equation*}
\cdots \stackrel{i_*}{\longrightarrow} \HSt_*(Y)\stackrel{\cdot Q}{\longrightarrow}\HSt_*(Y)\stackrel{p_*}{\longrightarrow} \HMt_*(Y)\stackrel{i_*}{\longrightarrow} \HSt_*(Y)\stackrel{j_*}{\longrightarrow} \cdots
\end{equation*}
This is a sequence of $\Rin$-modules where $V$ acts on $\HMt_*$ as $U^2$ and $Q$ acts as zero. 
\\
\par
For any rational number $d$ let $\V_d$ and  $\V^+_d$ be the graded $\ztwo [V]$-modules $\ztwo[V^{-1},V]$ and $\ztwo[V^{-1},V]/V\ztwo [V]$ where the grading is shifted so that the element $1$ has degree $d$. The following are the analogues of Proposition \ref{S3} and \ref{rational}.
\begin{prop}\label{S3pin}
We have the identifications as absolutely graded $\Rin$-modules:
\begin{align*}
\HSt_{*}(S^3)&\cong \V^+_2\oplus \V^+_1\oplus \V^+_0\\
\HSf_{*}(S^3)&\cong \Rin\langle-1\rangle\\\
\HSb_{*}(S^3)&\cong \V_2\oplus \V_1\oplus \V_0.
\end{align*}
The action of $Q$ on the to and bar groups is an isomorphism from the first tower onto the second tower and from the second tower onto the third (and zero otherwise). More generally, given a rational homology sphere $Y$ and a self-conjugate spin$^c$ structure $\spin$ we have an isomorphism of $\Rin$-modules
\begin{equation*}
\HSb_*(Y,\spin)\cong \HSb_*(S^3)
\end{equation*}
up to grading shift. The group $\HSt_*(Y,\spin)$ vanishes in degrees low enough, and the map $i_*$ is an isomorphism in degrees high enough.
\end{prop}

As in the usual case, a cobordism $W$ equipped with a self-conjugate spin$^c$ structure $\spin$ between $(Y_0,\spin_0)$ and $(Y_1,\spin_1)$ induces a map of $\Rin$-modules
\begin{equation*}
\HSt_*(W,\spin):\HSt_*(Y_0,\spin_0)\rightarrow \HSt_*(Y_1,\spin_1).
\end{equation*}
Notice that the restrictions $\spin_0$ and $\spin_1$ are also self-conjugate. When dealing with the total map, we need to consider the total completed Floer group
\begin{equation*}
\HSt_{\bullet}(Y)=\bigoplus_{[\spin]\in \mathrm{Spin}^c(Y)/\jmath} \HSt_{\bullet}(Y,[\spin])
\end{equation*}
when for a non self-conjugate pair $\spin\neq \bar{\spin}$ we define $\HSt_{\bullet}(Y,[\spin])$ to be the canonically identified groups
\begin{equation*}
\HMt_{\bullet}(Y,\spin)\equiv\HMt_{\bullet}(Y,\bar{\spin}).
\end{equation*}
Then \textit{any} cobordism $W$ from $Y_0$ to $Y_1$ induces a map
\begin{equation*}
\HSt_{\bullet}(W):\HSt_{\bullet}(Y_0)\rightarrow \HSt_{\bullet}(Y_1),
\end{equation*}
and these satisfy the analogue of Proposition \ref{functoriality}. Of course the whole discussion works for the from and bar versions of Floer homology.

\begin{Exercise}Suppose we have cobordisms $W_1$ from $Y_0$ to $Y_1$ and $W_2$ from $Y_1$ to $Y_2$ together with self-conjugate spin$^c$ structures that agree on $Y_1$. Does this imply that they glue to a self conjugate spin$^c$ structure on $W_2\circ W_1$? 
\end{Exercise}

\vspace{0.5cm}

\textbf{The construction. }The Seiberg-Witten equations have an extra symmetry when the spin$^c$ structure is actually induced by a spin structure.  Indeed, the choice of a spin structure provides us with two extra data:
\begin{itemize}
\item a preferred base spin$^c$ connection $B_0$ with $B_0^t$ flat, the spin connection;
\item a quaternionic structure $j$ on the Clifford bundle $S$.
\end{itemize}
Accordingly, we will also write the multiplication by complex numbers from the right. These two features are compatible in the sense that the Dirac operator $D_{B_0}$ is quaternionic linear, i.e.
\begin{equation*}
D_{B_0}(\Psi\cdot j)=(D_{B_0}\Psi)\cdot j
\end{equation*}
for every spinor $\Psi$. In this case the configuration space $\mathcal{C}(Y,\spin)$ comes with diffeomorphism $\jmath$ given by
\begin{equation*}
\jmath\cdot(B_0+b,\Psi)=(B_0-b,\Psi\cdot \jmath).
\end{equation*}
Then $\jmath^2$ is the identity on the connection component and minus the identity on the spinor components. As multiplication by $-1$ is a gauge transformation, $\jmath$ is an involution on the moduli space of configurations $\mathcal{B}(Y,\spin)$. Its only fixed points are the equivalence classes $[B,0]$ with $B$ the spin connection of one of the $2^{b_1(Y)}$ spin structures inducing the spin$^c$ structure $\spin$ (see Proposition \ref{spincself}). Similarly, there is an induced involution (still denoted by $\jmath$) on the blown-up moduli space of configurations $\mathcal{B}^{\sigma}(Y,\spin)$ which is fixed point-free.
\\
\par
If we choose $B_0$ as the basepoint, the Chern-Simons-Dirac functional $\mathcal{L}$ is $\jmath$-invariant. The goal is to perform the constructions we have discussed so far in a way such that the extra symmetry $\jmath$ is preserved. This causes some extra complications because this action has fixed points (downstairs). In particular, if we add equivariant perturbations, the perturbations of the operator $D_{B_0}$ with $B_0$ the spin connection will always be quaternionic linear. This implies that its eigenspaces are even dimensional, so that the condition of simple spectrum we have assumed so far cannot be satisfied and transversality is not achieved. On the other hand we can assure that this has two (complex) dimensional eigenspaces. In particular the description of Lemma \ref{critblow} works as follows in our case.
\begin{lem}For a generic perturbation $\jmath$-equivariant perturbation, the critical points of $(\mathrm{grad}\mathcal{L})^{\sigma}$ are of the following three types:
\begin{itemize}
\item irreducible critical points, which are transverse in the usual sense;
\item reducible critical points blowing down to a reducible point $[B,0]$ which is not a fixed point, which are trasverse in the usual sense;
\item reducible critical points blowing down to a fixed point, which come in submanifolds diffeomorphic to $\mathbb{C}P^1$ and are transverse in the Morse-Bott sense (i.e. the Hessian is non degenerate in the normal directions).
\end{itemize}
The fixed point free involution $\jmath$ sends irreducible configurations to irreducible configurations, the tower of critical points over $[B,0]$ (a non-fixed critical point) to the tower over $[\bar{B},0]$ (respecting the eigenvalues) and acts as the antipodal map on each of the critical submanifolds over the fixed points.
\end{lem}

\begin{proof}
Transversality is an issue only at the fixed points of the action downstairs, so that the only issue is at the spin connections. In this case, we can arrange generically that the eigenspaces are two dimensional (one dimensional quaternionic vector spaces). As in Exercise \ref{crit}, the critical points are then just the projectivizations of the eigenspaces, which are diffeomorphic to $\mathbb{C}P^1$. It is straighfroward to identify the action of $\jmath$.
\end{proof}
\begin{Exercise}
Check in the finite dimensional model that when $L$ is quaternionic linear the critical submanifolds are Morse-Bott. Of course, one has to assume that the underlying space is quaternionic.
\end{Exercise}
\vspace{0.5cm}

\textbf{Morse-Bott homology.} We need a framework to define homology via Morse-Bott functions. This has many technical complication in our setting but the main idea is very neat and we sketch it in the (closed) finite dimensional setting. In this case it was introduced in \cite{Fuk}, see also \cite{Hut}. Consider a Morse-Bott function $f$ on $M$. The underlying vector space of the Morse-Bott chain complex is the direct sum of (some variants of) the singular chain complexes of the critical submanifolds
\begin{equation*}
C_*(M,f)=\bigoplus_{\mathrm{C}} C_*(\mathrm{C}).
\end{equation*}
The differential combines the singular differential of each summand together with terms involving different critical submanifolds. In particular given critical submanifolds $\mathrm{C}_{\pm}$, there are evaluation maps on the compactified moduli spaces of trajectories connecting them
\begin{equation*}
\mathrm{ev}_{\pm}:\breve{M}^+(\mathrm{C}_-,\mathrm{C}_+)\rightarrow \mathrm{C}_{\pm}
\end{equation*}
sending a trajectory to its limit points. Then a singular chain
\begin{equation*}
f:\sigma\rightarrow \mathrm{C}_-
\end{equation*}
gives rise (under suitable transversality hypothesis) to the chain
\begin{equation*}
\mathrm{ev_+}: \sigma\tilde{\times} \breve{M}^+(\mathrm{C}_-,\mathrm{C}_+)\rightarrow \mathrm{C}_+
\end{equation*}
where the underlying space is the fibered product under the maps $f$ and $\mathrm{ev}_-$. The total differential of $\sigma$ is then defined to be the sum of its singular differential and all these chains obtained via fibered products. The proofs in this new framework carry over with the same formulas as the usual one: identities relating zero dimensional moduli spaces coming from boundaries of one dimensional spaces are now identities (at the chain level) of chains arising as the codimension one strata of fibered products as above. Of course, we need to consider chains $\sigma$ in some class of geometric objects so that the fibered products with the moduli spaces remain in that class.
\begin{exm}\label{morsebott}
Consider on $\mathbb{C}^3$ an operator $L$ with only two distinct eigenvalues $\lambda>\mu$, the latter having multiplicity two. The critical points on $\mathbb{C}P^2$ are then given by a point of index $4$ and a copy of $\mathbb{C}P^1$ with index $0$. We can apply our construction of Morse-Bott homology to compute the homology of $\mathbb{C}P^2$ using this function. The explicit computation is tricky because now the chain complex is infinite dimensional. On the other hand the chain complex $C_*(M,f)$ has a natural filtration provided by the value of $f$. The associate spectral sequence has $E^1$ page given by the homology of the critical submanifolds. In our case there cannot be higher differentials for grading reasons, so we recover the standard result: the homology is $\ztwo$ is degrees $0,2,4,6$ and zero otherwise.
\end{exm}

\vspace{0.5cm}
We can apply the construction we have just sketched to the case we are interested in. We obtain three Morse-Bott chain complexes $\check{C}_*(Y,\spin),\hat{C}_*(Y,\spin),\bar{C}_*(Y,\spin)$. These compute the usual monopole Floer homology groups. Furthermore, they carry a chain involution $\jmath$ given by applying the involution $\jmath$ to the chains generating the complexes. The new Floer groups (\ref{pin2}) are defined as the homology of the invariant subcomplexes.

\begin{exm}We compute $\HSt_*(S^3)$, following Example \ref{exS3}. In this new setting, the infinite tower of points is replaced by an infinite tower of $S^2$ on which $\jmath$ acts as the antipodal map. The relative grading of points in consecutive critical submanifolds is $4$ (see Exercise \ref{traj}). To compute the homology, we can apply the same trick as in Example \ref{morsebott} (with the function $\Lambda$ in place of $f$). The key point is that this filtration is invariant under the action of $\jmath$. In particular the $E^1$ page of the associated spectral sequence is the direct sum of the the homologies of the invariants of the chain complexes of the critical submanifolds. These can be easily identified with the homology of the quotient, which is just a copy of $\mathbb{R}P^2$. From this discussion Proposition \ref{S3pin} follows.
\end{exm}

\vspace{1cm}
\section{The correction terms and The Triangulation conjecture}

Manolescu introduced his new invariants in order to solve the almost one-hundred-year-old Triangulation conjecture. This asserted that every topological manifold is homeomorphic to a simplicial complex. This was already known to be true in dimension at most three and false in dimension four. We refer the reader to \cite{Man3} and \cite{Man4} for a more thorough discussion of the background. Here we discuss how the disproof can be carried over in our setting (in a formally analogous way). The first key step is a classic result of Galeski-Stern and Matumoto, reducing the problem to a low-dimensional one.
\begin{thm}[\cite{Gal}, \cite{Mat}]\label{80s}
The Triangulation Conjecture is false in every dimension $\geq 5$ if and only if there are no order two elements with Rokhlin invariant one in the homology cobordism group $\Theta^H_3$.
\end{thm} 
Recall that the homology cobordism group $\Theta^H_3$ is the group whose elements are oriented homology three-spheres up to homology cobordism. The addition is given by connected sum and the inverse is obtained by reversing the orientation. The definition works in every dimension (using PL spheres) but all the others are trivial by a result of Kervaire. One way to see that $\Theta^H_3$ is not trivial is via the Rokhlin invariant, which is a homomorphism
\begin{equation*}
\Theta^H_3\rightarrow\mathbb{Z}/2\mathbb{Z}.
\end{equation*}
This is obtained by sending a homology sphere $Y$ to $\sigma(W)/8$ where $W$ is any spin manifold whose boundary is $Y$. This is well defined by Rokhlin's signature theorem (see Example \cite{Kir}) and surjective (see Exercise \ref{E8}).
\\
\par
Manolescu's disproof goes through the following existence theorem.
\begin{thm}[\cite{Man2}]\label{beta}
There exists an integer valued map $\beta$ on $\Theta^H_3$ such that:
\begin{itemize}
\item $\beta$ is an integral lift of the Rokhlin invariant;
\item $\beta(-Y)$ is $-\beta(Y)$.
\end{itemize}
\end{thm}
\begin{proof}[Disproof of the Triangulation Conjecture in dimension at least $5$.]
We prove the equivalent condition provided by Theorem \ref{80s}. If a homology sphere has order two, $[Y]$ is the same as $[-Y]$. The second property of $\beta$ in Theorem \ref{beta} implies that
\begin{equation*}
\beta([Y])=\beta([-Y])=-\beta([Y])
\end{equation*}
so $\beta([Y])$ is zero. Hence also the Rokhlin invariant is zero by the first property.
\end{proof}

\vspace{0.3cm}
Notice the analogy between Theorem \ref{beta} and Corollary \ref{froy}. Indeed, $\beta$ arises as an analogue in the $\Pin$ theory of the Fr\o yshov invariant. Recall the computation of \textit{HS-bar} for a homology sphere $Y$ in Proposition \ref{S3pin}. Its image via $i_*$ in $\HSt_{\bullet}(Y,\spin)$, considered as an $\ztwo[[V]]$-module, decomposes as the direct sum of three towers
\begin{equation*}
\V^+_c\oplus \V^+_b\oplus \V^+_a.
\end{equation*}
The action of $Q$ sends the first tower onto the second and the second tower onto the third. Manolescu's correction terms are then defined to be numbers
\begin{equation*}
\alpha(Y)\geq\beta(Y)\geq\gamma(Y),
\end{equation*}
such that
\begin{equation*}
a=2\alpha(Y),\quad b= 2\beta(Y)+1, \quad c=2\gamma(Y)+2.
\end{equation*}
The inequalities between these quantities follow from the module structure.
\begin{rem} 
It is important to remark that the quantities $\alpha, \beta$ and $\gamma$ are a priori different different from those defined by Manolescu in \cite{Man2}. On the other hand, we call them Manolescu's correction terms as they are conjecturally equivalent to those arising from his (formally equivalent) construction.
\end{rem}
The following is the key result (from which Theorem \ref{beta} readily follows).
\begin{prop}
Manolescu's correction terms $\alpha,\beta,\gamma$ satisfy the following properties.
\begin{enumerate}
\item They are integral lifts of the Rokhlin invariant.
\item	$\alpha(-Y)$, $\beta(-Y)$ and $\gamma(-Y)$ are respectively $-\gamma(Y)$, $-\beta(Y)$ and $-\alpha(Y)$.
\item If $W$ is a negative definite spin cobordism between $Y_0$ and $Y_1$, we have the inequalities
\begin{align*}
\alpha(Y_1)&\geq \alpha(Y_0)+b_2(W)/8\\
\beta(Y_1)&\geq \beta(Y_0)+b_2(W)/8\\
\gamma(Y_1)&\geq \gamma(Y_0)+b_2(W)/8
\end{align*}
In particular, they are invariant under homology cobordism.
\end{enumerate}
\end{prop}
\begin{proof}
We sketch the proofs of these results. For part $(1)$, following the proof of Proposition \ref{gradingprop}, we now have that $c_1^2$ is zero (because $\spin$ is torsion) and the real index of the Dirac operator is divisible by four (because it is quaternionic linear). Hence the absolute gradings is congruent modulo four to $-\sigma(W)/4$, and the result follows. For the part $(2)$, the key observation is that via Poincar\'e duality the tower $\V^+_a$ defining $\alpha(Y)$ becomes the one defining $\gamma(-Y)$, and similarly for the other towers. The final statement follows as in Proposition \ref{fro}. After surgery, we can assume that $b_1=0$ and the manifold is spin. The relevant reducible moduli spaces are now copies of $\mathbb{C}P^1$ corresponding to the projectivization of the kernel of the quaternionic operator $D_{A_0}^+$. They map diffeomorphically onto the respective critical manifolds, hence the induced map
\begin{equation*}
\HSb_*(W,\spin):\HSb_*(Y_0)\rightarrow \HSb_*(Y_1)
\end{equation*}
is an isomorphism of degree $b_2(W)/4$.
\end{proof}

\vspace{0.5cm}
\textbf{Acknowledgements.} I would like to express my gratitude to the organizers of the G\"okova Geometry-Topology Conference 2015, and in particular Selman Akbulut. The idea of these notes started following an invitation of Andras Stipsicz, to give a lecture series at The Renyi Institute in Budapest in Summer 2016. I would like to thank both of them and the anonymous referee for the helpful comments and suggestions. A preliminary version of them was conceived for a lecture series given at the University of Pisa in Fall 2014, and I would like to thank them, and in particular Bruno Martelli, for their hospitality. Finally, I would like to thank my advisor Tom Mrowka for introducing me to this beautiful subject. This work was partially supported by the NSF grant DMS-0805841.

\end{document}

%% file: blowupflows.pdf_tex
\begingroup%
  \makeatletter%
  \providecommand\color[2][]{%
    \errmessage{(Inkscape) Color is used for the text in Inkscape, but the package 'color.sty' is not loaded}%
    \renewcommand\color[2][]{}%
  }%
  \providecommand\transparent[1]{%
    \errmessage{(Inkscape) Transparency is used (non-zero) for the text in Inkscape, but the package 'transparent.sty' is not loaded}%
    \renewcommand\transparent[1]{}%
  }%
  \providecommand\rotatebox[2]{#2}%
  \ifx\svgwidth\undefined%
    \setlength{\unitlength}{425.19685039bp}%
    \ifx\svgscale\undefined%
      \relax%
    \else%
      \setlength{\unitlength}{\unitlength * \real{\svgscale}}%
    \fi%
  \else%
    \setlength{\unitlength}{\svgwidth}%
  \fi%
  \global\let\svgwidth\undefined%
  \global\let\svgscale\undefined%
  \makeatother%
  \begin{picture}(1,0.53333333)%
    \put(0,0){\includegraphics[width=\unitlength,page=1]{blowupflows.pdf}}%
  \end{picture}%
\endgroup%

%% file: obstructed.pdf_tex
\begingroup%
  \makeatletter%
  \providecommand\color[2][]{%
    \errmessage{(Inkscape) Color is used for the text in Inkscape, but the package 'color.sty' is not loaded}%
    \renewcommand\color[2][]{}%
  }%
  \providecommand\transparent[1]{%
    \errmessage{(Inkscape) Transparency is used (non-zero) for the text in Inkscape, but the package 'transparent.sty' is not loaded}%
    \renewcommand\transparent[1]{}%
  }%
  \providecommand\rotatebox[2]{#2}%
  \ifx\svgwidth\undefined%
    \setlength{\unitlength}{400bp}%
    \ifx\svgscale\undefined%
      \relax%
    \else%
      \setlength{\unitlength}{\unitlength * \real{\svgscale}}%
    \fi%
  \else%
    \setlength{\unitlength}{\svgwidth}%
  \fi%
  \global\let\svgwidth\undefined%
  \global\let\svgscale\undefined%
  \makeatother%
  \begin{picture}(1,0.74)%
    \put(0,0){\includegraphics[width=\unitlength]{obstructed.pdf}}%
    \put(0.30685291,-0.00604883){\color[rgb]{0,0,0}\makebox(0,0)[lb]{\smash{$\partial\mathcal{B}^{\sigma}(Y,\spin)$}}}%
    \put(0.75297485,0.63834716){\color[rgb]{0,0,0}\makebox(0,0)[lb]{\smash{$\mathrm{ind}=2$}}}%
    \put(0.7388327,0.12316931){\color[rgb]{0,0,0}\makebox(0,0)[lb]{\smash{$\mathrm{ind}=0$}}}%
    \put(0.12943153,0.52318969){\color[rgb]{0,0,0}\makebox(0,0)[lb]{\smash{$\mathrm{ind}=1$}}}%
    \put(0.13545186,0.20794091){\color[rgb]{0,0,0}\makebox(0,0)[lb]{\smash{$\mathrm{ind}=1$}}}%
  \end{picture}%
\endgroup%

%% file: disc.pdf_tex
\begingroup%
  \makeatletter%
  \providecommand\color[2][]{%
    \errmessage{(Inkscape) Color is used for the text in Inkscape, but the package 'color.sty' is not loaded}%
    \renewcommand\color[2][]{}%
  }%
  \providecommand\transparent[1]{%
    \errmessage{(Inkscape) Transparency is used (non-zero) for the text in Inkscape, but the package 'transparent.sty' is not loaded}%
    \renewcommand\transparent[1]{}%
  }%
  \providecommand\rotatebox[2]{#2}%
  \ifx\svgwidth\undefined%
    \setlength{\unitlength}{283.46456693bp}%
    \ifx\svgscale\undefined%
      \relax%
    \else%
      \setlength{\unitlength}{\unitlength * \real{\svgscale}}%
    \fi%
  \else%
    \setlength{\unitlength}{\svgwidth}%
  \fi%
  \global\let\svgwidth\undefined%
  \global\let\svgscale\undefined%
  \makeatother%
  \begin{picture}(1,0.7)%
    \put(0,0){\includegraphics[width=\unitlength,page=1]{disc.pdf}}%
    \put(0.91843188,0.30802534){\color[rgb]{0,0,0}\makebox(0,0)[lb]{\smash{$h$}}}%
    \put(0,0){\includegraphics[width=\unitlength,page=2]{disc.pdf}}%
  \end{picture}%
\endgroup%

%% file: spectralflow.pdf_tex
\begingroup%
  \makeatletter%
  \providecommand\color[2][]{%
    \errmessage{(Inkscape) Color is used for the text in Inkscape, but the package 'color.sty' is not loaded}%
    \renewcommand\color[2][]{}%
  }%
  \providecommand\transparent[1]{%
    \errmessage{(Inkscape) Transparency is used (non-zero) for the text in Inkscape, but the package 'transparent.sty' is not loaded}%
    \renewcommand\transparent[1]{}%
  }%
  \providecommand\rotatebox[2]{#2}%
  \ifx\svgwidth\undefined%
    \setlength{\unitlength}{595.27559055bp}%
    \ifx\svgscale\undefined%
      \relax%
    \else%
      \setlength{\unitlength}{\unitlength * \real{\svgscale}}%
    \fi%
  \else%
    \setlength{\unitlength}{\svgwidth}%
  \fi%
  \global\let\svgwidth\undefined%
  \global\let\svgscale\undefined%
  \makeatother%
  \begin{picture}(1,0.28571429)%
    \put(0,0){\includegraphics[width=\unitlength,page=1]{spectralflow.pdf}}%
    \put(0.72911164,0.08782037){\color[rgb]{0,0,0}\makebox(0,0)[lb]{\smash{$t$}}}%
    \put(0,0){\includegraphics[width=\unitlength,page=2]{spectralflow.pdf}}%
    \put(0.31582009,0.1401512){\color[rgb]{0,0,0}\makebox(0,0)[lb]{\smash{$0$}}}%
  \end{picture}%
\endgroup%

%% file: stretch.pdf_tex
\begingroup%
  \makeatletter%
  \providecommand\color[2][]{%
    \errmessage{(Inkscape) Color is used for the text in Inkscape, but the package 'color.sty' is not loaded}%
    \renewcommand\color[2][]{}%
  }%
  \providecommand\transparent[1]{%
    \errmessage{(Inkscape) Transparency is used (non-zero) for the text in Inkscape, but the package 'transparent.sty' is not loaded}%
    \renewcommand\transparent[1]{}%
  }%
  \providecommand\rotatebox[2]{#2}%
  \ifx\svgwidth\undefined%
    \setlength{\unitlength}{595.27559055bp}%
    \ifx\svgscale\undefined%
      \relax%
    \else%
      \setlength{\unitlength}{\unitlength * \real{\svgscale}}%
    \fi%
  \else%
    \setlength{\unitlength}{\svgwidth}%
  \fi%
  \global\let\svgwidth\undefined%
  \global\let\svgscale\undefined%
  \makeatother%
  \begin{picture}(1,0.30952381)%
    \put(0,0){\includegraphics[width=\unitlength,page=1]{stretch.pdf}}%
    \put(0.72765191,0.04610247){\color[rgb]{0,0,0}\makebox(0,0)[lb]{\smash{$Y_2$}}}%
    \put(0.1858765,0.04596598){\color[rgb]{0,0,0}\makebox(0,0)[lb]{\smash{$Y_0$}}}%
    \put(0,0){\includegraphics[width=\unitlength,page=2]{stretch.pdf}}%
    \put(0.2428667,0.23480323){\color[rgb]{0,0,0}\makebox(0,0)[lb]{\smash{$W_1$}}}%
    \put(0.65841636,0.23075787){\color[rgb]{0,0,0}\makebox(0,0)[lb]{\smash{$W_2$}}}%
    \put(0.41829205,0.18726132){\color[rgb]{0,0,0}\makebox(0,0)[lb]{\smash{$[-T,T]\times Y_1$}}}%
  \end{picture}%
\endgroup%